\documentclass[letterpaper,11pt]{amsart}
\usepackage{indentfirst} 
\usepackage{amssymb}
\usepackage{mathrsfs}
\usepackage{amsmath}
\usepackage{amsthm}
\usepackage{thmtools}
\usepackage{enumitem}
\usepackage{mathtools}
\usepackage{bbm}            
\usepackage[colorlinks = true, linkcolor = black, citecolor = blue ]{hyperref}
\usepackage[usenames,dvipsnames]{xcolor} 
\usepackage{tikz}
\usepackage[left=3cm, right=3cm, bottom=3.4cm]{geometry}

\theoremstyle{plain}
\newtheorem{theorem}{Theorem}[section]

\theoremstyle{plain}
\newtheorem{proposition}{Proposition}[subsection]
\newtheorem{lemma}[proposition]{Lemma}
\newtheorem{corollary}[proposition]{Corollary}

\theoremstyle{definition}
\newtheorem{definition}{Definition}[subsection]

\theoremstyle{remark}
\newtheorem{remark}{Remark}[subsection]

\declaretheorem[name=Acknowledgements,numbered=no]{ack}

\newcommand{\vertiii}[1]{{\left\vert\kern-0.25ex\left\vert\kern-0.25ex\left\vert #1 
    \right\vert\kern-0.25ex\right\vert\kern-0.25ex\right\vert}}

\def\e{\epsilon}
\def\R{{\mathbb R}}
\def\N{{\mathbb N}}

\def\C{{\mathscr C}}
\def\L{{\mathscr L}}
\def\T{{\mathbf T}}
\def\z{{\mathbf z}}

\begin{document}

\title[The Vlasov--Poisson system with a trapping potential]{Modified scattering of small data solutions to the Vlasov--Poisson system with a trapping potential}

\author[L\'eo Bigorgne]{L\'eo Bigorgne} \address{Institut de recherche math\'ematique de Rennes (IRMAR) - UMR 6625, CNRS, Universit\'e de Rennes, F-35000 Rennes, France.}
\email{leo.bigorgne@univ-rennes1.fr}

\author[Anibal Velozo Ruiz]{Anibal Velozo Ruiz} \address{Facultad de matem\'atica, Pontificia Universidad Cat\'olica de Chile, Avenida Vicu\~na Mackenna 4860, Santiago, Chile.}
\email{apvelozo@mat.uc.cl}

\author[Renato Velozo Ruiz]{Renato Velozo Ruiz} \address{Laboratory Jacques-Louis Lions (LJLL), University Pierre and Marie Curie (Paris 6), 4
place Jussieu, 75252 Paris, France.}\address{Department of Mathematics, University of Toronto, 40 St. George Street, Toronto, ON, Canada.}
\email{renato.velozo.ruiz@utoronto.ca}

\begin{abstract}
In this paper, we study small data solutions to the Vlasov--Poisson system with the simplest external potential, for which unstable trapping holds for the associated Hamiltonian flow. First, we provide a new proof of global existence for small data solutions to the Vlasov--Poisson system with the trapping potential $\frac{-|x|^2}{2}$ in dimension two. We exploit the uniform hyperbolicity of the Hamiltonian flow, by making use of the commuting vector fields contained in the stable and unstable invariant distributions of phase space for the linearized system. In contrast with the proof in \cite{VV23}, we do not use modified vector field techniques. Moreover, we obtain small data modified scattering for this non-linear system. We show that the distribution function converges to a new regular distribution function along modifications to the characteristics of the linearized problem. We define the linearly growing corrections to the characteristic system, in terms of a precise effective asymptotic force field. We make use of the scattering state to obtain the late-time asymptotic behavior of the spatial density. Finally, we prove that the distribution function (up to normalization) converges weakly to a Dirac mass on the unstable manifold of the origin.
\end{abstract}

\maketitle

\setcounter{tocdepth}{1}
\tableofcontents

\section{Introduction}\label{introduction_small_data_unstable_potential}
In this paper, we study the asymptotic behavior of collisionless many-particle systems on $\R^2_x$. We consider many-particle systems described statistically by a distribution function satisfying a collisionless non-linear model arising in kinetic theory. More precisely, we investigate the asymptotic properties of small data solutions $f(t,x,v)$ to the \emph{Vlasov--Poisson system with the potential $\frac{-|x|^2}{2}$}; given by 
\begin{equation}\label{vlasov_poisson_external_potential}
\begin{cases}
\partial_t f+v\cdot\nabla_xf+x\cdot\nabla_vf -\mu\nabla_x\phi \cdot \nabla_vf=0,\\ 
\Delta_x \phi=\rho(f),\\
\rho(f)(t,x):=\int_{\R_v^2}f(t,x,v) \mathrm{d} v,\\
f(t=0,x,v)=f_0(x,v),
\end{cases}
\end{equation}
where $t\in [0,\infty)$, $x\in \R^2_x$, $v\in \R^2_v$, and $\mu\in \{1,-1\}$ is a fixed constant. The interaction between the particles of the system is \emph{attractive} when $\mu=1$, or \emph{repulsive} when $\mu=-1$. The nonlinearity in this kinetic PDE system arises from the mean field generated by the many-particle system. The Vlasov--Poisson system with the external potential $\frac{-|x|^2}{2}$, describes a collisionless many-particle system for which the trajectories described by its particles are set by the mean field generated by the system, and the external potential $\frac{-|x|^2}{2}$. We call $\nabla_x\phi$ the \emph{force field}, and $\rho(f)$ the \emph{spatial density}.

The Vlasov--Poisson system is a non-linear PDE system whose dynamics have been extensively studied in the scientific literature. The first well-posedness results for this PDE system (without external potential) were obtained by Okabe and Ukai \cite{OU78}, who proved global well-posedness in dimension two and local well-posedness in dimension three. Seminal independent works by Pfaffelmoser \cite{Pf92} and Lions--Perthame \cite{LP91} proved \emph{global well-posedness} for the Vlasov--Poisson system (without external potential) in dimension three. See also Schaeffer's proof \cite{Sch91}. These global well-posedness results can be adapted to incorporate an external potential $\Phi(x)$, as long as $\nabla_x\Phi$ has Lipschitz regularity (see the introduction of \cite{GHK12}). Although the well-posedness properties of the Vlasov--Poisson system have been settled in \cite{Pf92, LP91, Sch91}, the description of the non-linear dynamics of the solutions to this PDE system for arbitrary finite energy data is not yet fully understood. 

The class of small data solutions for the Vlasov--Poisson system has been studied in great detail in the literature. The first small data asymptotic stability result for the Vlasov--Poisson system was obtained by Bardos and Degond \cite{BD85}, who studied the evolution in time of solutions to the Vlasov--Poisson system for compactly supported initial data using the method of characteristics. Later in time, this small data global existence result was improved by Hwang, Rendall and Vel\'asquez \cite{HRV11}, who proved optimal decay in time for higher order derivatives of the spatial density for compactly supported data, using again the method of characteristics. Subsequently, the stability of the vacuum solution for the Vlasov--Poisson system in \cite{BD85} was revisited by Smulevici \cite{Sm16}, who proved small data global existence based upon energy estimates using the vector field method. As a result, Smulevici \cite{Sm16} obtained boundedness in time of a suitable energy norm, and optimal space and time decay estimates for the spatial density. We emphasize the novel \emph{modified vector field technique} introduced in \cite{Sm16}, in order to address the small data global existence for the Vlasov--Poisson system in dimension three. Later Duan \cite{Du22} simplified the functional framework and the proof of the stability of vacuum for the Vlasov--Poisson system in \cite{Sm16}. See also the work by Wang \cite{W23} for another proof of the stability of vacuum for the Vlasov--Poisson system in dimension three via Fourier techniques.

Moreover, there have been several works concerned on the \emph{scattering properties of the distribution function} for small data solutions to the Vlasov--Poisson system on $\R^3_x\times\R^3_v$. First, Choi and Kwon \cite{CK16} proved that the distribution function converges to a new distribution function along modifications to the characteristics of the linearized problem, for small data solutions to the Vlasov--Poisson system. We refer to this property of the distribution function as \emph{modified scattering}. Later, Ionescu, Pausader, Wang, and Widmayer \cite{IPWW}, obtained a new proof of small data modified scattering using methods inspired from dispersive analysis. The work \cite{IPWW} identified an explicit correction to the characteristic system, in terms of an effective asymptotic force field defined using a normalized mass for every energy level $\{v=v_0\}$ for $v_0\in\R^3_v$. Around the same time, Pankavich \cite{P22} proved modified scattering for a multispecies collisionless plasma assuming that the electric field decays sufficiently fast (instead of assuming smallness of the compactly supported initial data considered in \cite{P22}). The work \cite{P22} also identifies precise self-similar asymptotic profiles for the spatial density and the electric field. 

More recently, the first author has established small data modified scattering for the \emph{relativistic Vlasov--Maxwell system} on $\R^3_x\times \R^3_v$ \cite{B22}, which models the dynamics of a collisionless plasma of charged particles. We observe that the small data modified scattering result in \cite{B22} does not require smallness on the Maxwell field. The strategy used to obtain the stability of vacuum for the Vlasov--Maxwell system simplifies previous vector field methods to address small data global existence for the classical Vlasov-type systems in dimension three. The proof in \cite{B22} is obtained through a commuting vector field approach which \emph{does not} require modified vector field techniques to obtain small data global existence. Moreover, \cite{B22} establishes small data modified scattering to a new \emph{highly regular} distribution function along modifications of the characteristics of the free relativistic transport equation. This part of the proof requires the introduction of novel \emph{asymptotic modified vector fields}, whose components depend on an effective asymptotic Lorentz force for the characteristic system. See the work of Pankavich and Ben-Artzi \cite{PB23} for an alternative proof of small data modified scattering for the relativistic Vlasov--Maxwell system for compactly supported initial data. 

The small data modified scattering results \cite{B22} and \cite{PB23} for the relativistic Vlasov--Maxwell system on $\R^3_x\times\R^3_v$ have been obtained after several proofs of small data global existence have been established in the literature. The vector field method for collisionless systems was used by the first author \cite{B20, B21, Big22}, in order to prove the stability of vacuum for the relativistic Vlasov--Maxwell in dimension greater or equal than three. Wang \cite{Wa22} obtained another proof of the stability of vacuum for this system in dimension three, by combining the vector field method and Fourier techniques. Small data global existence for the relativistic Vlasov--Maxwell system was first shown by Glassey and Schaeffer \cite{GS87} using the method of characteristics.

The motivation behind considering small data solutions to the Vlasov--Poisson system with the potential $\frac{-|x|^2}{2}$ comes from stability results for \emph{dispersive collisionless systems} for which the dynamics described by their particles is \emph{hyperbolic}. We consider the Vlasov--Poisson system with the simplest external potential for which \emph{unstable trapping} holds for the Hamiltonian flow associated to small data solutions of this system. We say that \emph{unstable trapping} holds for a Hamiltonian flow in $\R^2_x\times \R^2_v$, if the trajectories of the flow escape to infinity for every point in phase space, except for a non-trivial set of measure zero for which the future of every trajectory of the flow is bounded. The second two authors \cite{VV23} have recently established small data global existence for the Vlasov--Poisson system with the potential $\frac{-|x|^2}{2}$ in dimension two or higher. Furthermore, it was proved that unstable trapping holds for the Hamiltonian flow associated to small data solutions of this system. In fact, \cite[Theorem 1.3]{VV23} provides an explicit teleological construction of the trapped set in terms of the non-linear evolution of the force field in dimension two or higher. The proofs of \cite{VV23} exploit the uniform hyperbolicity of the Hamiltonian flow by making use of the commuting vector fields contained in the stable and unstable invariant distributions of phase space\footnote{We refer to a \emph{distribution} in phase space $\R^n_x\times\R^n_v$ as a map $(x,v)\mapsto \Delta_{(x,v)}\subseteq T_{(x,v)}(\R^n_x\times\R^n_v)$, where $\Delta_{(x,v)}$ are vector subspaces satisfying suitable conditions (in the standard sense used in differential geometry).} for the linearized system. In the specific case of dimension two, \cite{VV23} makes use of modified vector field techniques due to the slow decay estimates in time, which suggests that small data modified scattering holds for the Vlasov--Poisson system with the potential $\frac{-|x|^2}{2}$ in dimension two. 

In this paper, we prove small data modified scattering for the Vlasov--Poisson system with the potential $\frac{-|x|^2}{2}$ in dimension two. Firstly, we provide a new proof of small data global existence for solutions to the Vlasov--Poisson system with the potential $\frac{-|x|^2}{2}$ in dimension two. In contrast with the proof of \cite[Theorem 1.2]{VV23}, we do not use modified vector field techniques to establish small data global existence. Later, we obtain small data modified scattering for this non-linear system. We show that the distribution function converges to a new highly regular distribution function along modifications to the characteristics of the linearized problem. We define the linearly growing corrections to the characteristic system in terms of a precise effective asymptotic force field. The regularity of the scattering state is proven by introducing novel asymptotic modified vector fields, whose components depend on the effective asymptotic force field for the characteristic system. Later, we make use of the scattering state to obtain the late-time asymptotic behavior of the spatial density. Finally, we prove that the distribution function (up to normalization) converges weakly to a Dirac mass on the unstable manifold of the origin. The mass of the corresponding Dirac mass is explicitly identified in terms of the scattering state.

We investigate this model with the hope to offer new insights to study asymptotic stability results for dispersive collisionless systems for which the associated characteristic system is hyperbolic. This dispersive behavior holds locally for 1D Hamiltonian flows arising from potentials with a global maximum in a neighborhood of the associated hyperbolic fixed point. An important example of dispersive collisionless systems for which the associated Hamiltonian flow is hyperbolic is given by collisionless systems in the exterior of black hole backgrounds which admit a \emph{normally hyperbolic trapped set}. We have in mind massless collisionless systems on the exterior of black holes spacetimes, as for example, the subextremal family of Kerr black holes. See \cite{WZ11, D15} for more details.

\subsection{A first glance to the main results}
In this manuscript, we investigate the non-linear dynamics of small data solutions to the Vlasov--Poisson system with the external potential $\frac{-|x|^2}{2}$ in dimension two, given by
\begin{equation}\label{vlasov_poisson_unstable_trapping_potential_paper}
\begin{cases}
\partial_t f+v\cdot\nabla_xf+x\cdot \nabla_vf-\mu\nabla_x\phi \cdot \nabla_vf=0,\\
\Delta_x \phi=\rho(f),\\
\rho(f)(t,x):=\int_{\R_v^2}f(t,x,v)\mathrm{d}v,\\
f(t=0,x,v)=f_0(x,v),
\end{cases}
\end{equation}
where $t\in [0,\infty)$, $x\in \R^2_x$, $v\in \R^2_v$, and $\mu\in\{1,-1\}$ is a fixed constant. 

The local well-posedness theory for this PDE system is standard (see for instance \cite[Section 3]{HK19}). We study the evolution in time of small initial distribution functions $f_0:\R^{2}_x\times \R^{2}_v\to [0,\infty)$, in a space of functions defined by a higher order weighted $L^{\infty}_{x,v}$ norm $$ \vertiii{f_0}_{N,M}:=\sum_{|\beta|+|\kappa|\leq N}\sup_{(x,v)\in \R^2_x\times \R^2_v}\langle x-v \rangle^{M+|\beta|} \langle x+v \rangle^{M+|\kappa|}|(\partial_x-\partial_v)^{\beta}(\partial_x+\partial_v)^{\kappa} f_0|,$$ where $N,M \in \N$ and $\langle \cdot \rangle$ is the standard Japanese bracket. The differential operators considered in this norm are obtained as compositions of vector fields in a class ${\lambda}$ of commuting vector fields for the linearized system, at time $t=0$. Similarly, the weights considered in the norm above are conserved quantities along the characteristic flow of the linearized system, at time $t=0$. See Section \ref{preliminaries_unstable_potential} for further details concerning the commuting vector fields and the weights for the linearized problem.

In the following, we denote by $(X_{\L},V_{\L})$ to the components of the characteristics to the linearized problem, given by $$(X_{\L}(t,x,v),V_{\L}(t,x,v)):=(x\cosh t +v\sinh t ,x\sinh t+v\cosh t).$$ For a distribution function satisfying the linear Vlasov equation with the potential $\frac{-|x|^2}{2}$, we have that $f(t,X_{\L}(t),V_{\L}(t))=f_0(x,v)$, for every $t\in [0,\infty)$. In contrast, the expression $f(t,X_{\L}(t,x,v),V_{\L}(t,x,v))$ does not converge as $t\to \infty$, for small data solutions to the Vlasov--Poisson system with the potential $\frac{-|x|^2}{2}$ in dimension two, unless the initial data is identically zero. 

We introduce the coordinate system $(s,u)$ in phase space, defined by $$s^i:=\frac{x^i-v^i}{2},\qquad u^i:=\frac{x^i+v^i}{2},$$ which is more suitable to capture the hyperbolicity of the linearized system. We will frequently use this identification without explicit reference. In the following, we write a distribution function in the hyperbolic coordinate system $(s,u)$ by $$\bar{f}(t,s,u):=f(t,x,v).$$

Similarly, in the rest of the paper, the notation $A\lesssim B$ is repetitively used to specify that there exists a universal constant $C > 0$ such that $A \leq CB$, where $C$ depends only on the corresponding order of regularity, or other fixed constants.

\begin{theorem}[Small data modified scattering for the Vlasov--Poisson system with the potential $\frac{-|x|}{2}$]\label{thm_main_result_first_version}
Every solution $f$ to the Vlasov--Poisson system with the potential $\frac{-|x|^2}{2}$ arising from smooth and small initial data is global in time. Moreover, the following properties hold.
\begin{enumerate}[label = (\alph*)]
\item Let $\bar{u}\in \R^2_u$. The normalized stable average of $f$ along $\{u=\bar{u}\}$ converges to a regular function $Q_{\infty}:\R^2_u\to \R$ such that $$\forall t\in [0,\infty), \qquad \Big|e^{2t}\int_{\R^2_s} \bar{f}(t,s,e^t\bar{u})\mathrm{d}s -Q_{\infty}(\bar{u}) \Big|\lesssim \dfrac{(1+t)^4}{e^{2t}}.$$
\item Let $\phi_{\mathrm{asymp}}:\R^2_u\to \R$ be the solution to $\Delta_u \phi_{\mathrm{asymp}}=Q_{\infty}$. The force field has a regular self-similar asymptotic profile $u\mapsto \nabla_u \phi_{\mathrm{asymp}}$, in the sense that $$\forall (t,x,v)\in [0,\infty)\times \R^2_x\times \R^2_v, \qquad |e^t \nabla_x\phi (t,X_{\L}(t,x,v))-
\nabla_u \phi_{\mathrm{asymp}} (u)|\lesssim \langle e^{-t}s \rangle \dfrac{(1+t)^7}{e^t}.$$ 
\item We have modified scattering for the distribution function, in the sense that there exists a regular distribution $f_{\infty}: \R^2_x\times\R^2_v\to [0,\infty)$, such that $$\forall (t,x,v)\in [0,\infty)\times \R^2_x\times \R^2_v, \quad |f(t,X_{\mathscr{C}}(t,x,v), V_{\C}(t,x,v))-f_{\infty}(x,v)|\lesssim \dfrac{(1+t)^{16}}{e^t},$$ where the components $(X_{\C}, V_{\C})$ of the modified characteristics are defined as \begin{align*} X_{\C}(t,x,v):=x\cosh t +v\sinh t+\dfrac{\mu t}{2e^{t}} \nabla_u \phi_{\mathrm{asymp}}(u),\\ V_{\C}(t,x,v):=x\sinh t+v\cosh t-\dfrac{\mu t}{2e^{t}}\nabla_u \phi_{\mathrm{asymp}}(u). \end{align*} 
\end{enumerate}
\end{theorem} 

\begin{remark}
\begin{enumerate}[label = (\alph*)]
\item The proof of Theorem \ref{thm_main_result_first_version} fits into the general framework of vector field methods for dispersive collisionless kinetic equations. In order to show small data modified scattering for the Vlasov--Poisson system with the potential $\frac{-|x|^2}{2}$, we follow the strategy outlined by \cite{B22}. We observe that the self-similar asymptotic profile $\nabla_u\phi_{\mathrm{asymp}}$ of the force field allows us to write an explicit correction to the characteristic system. The self-similar asymptotic profile $\nabla_u\phi_{\mathrm{asymp}}$ is defined as the field induced by the \emph{asymptotic Poisson equation} $\Delta_u \phi_{\mathrm{asymp}}=Q_{\infty}$, where the normalized stable average $Q_{\infty}$ plays the role of density on the unstable leaves of phase space. See Subsection \ref{subsection_self_similar_asymp_profile} for further details.  
\item We exploit the uniform hyperbolicity of the non-linear Hamiltonian flow by making use of the commuting vector fields contained in the stable and unstable invariant distributions of phase space for the linearized system. In contrast with the proof of \cite[Theorem 1.2]{VV23}, we do not use modified vector field techniques to show small data global existence. Nonetheles, the regularity of the scattering state is proven by introducing novel \emph{asymptotic modified vector fields}, whose components depend on the effective asymptotic force field for the characteristic system. The modifications to the commuting vector fields for the linearized system \emph{grow linearly in time}. This contrast with previous results \cite{CK16, IPWW, P22, B22, PB23} concerning small data modified scattering for collisionless kinetic equations, where the modifications \emph{grow logarithmically in time}.
\end{enumerate}
\end{remark}

The first part of Theorem \ref{thm_main_result_first_version} consists in proving small data global existence for the Vlasov--Poisson system with the potential $\frac{-|x|^2}{2}$. For this purpose, we prove exponential decay in time of velocity averages. Previously in \cite{VV23}, the second two authors proved that $|\rho(f)|\lesssim e^{-2t}$ for small data solutions of this non-linear system. As part of the proof of Theorem \ref{thm_main_result_first_version}, we establish the late time asymptotic behavior of the spatial density.

\begin{theorem}[Late-time asymptotic of the spatial density]
For every solution $f$ to the Vlasov--Poisson system with the potential $\frac{-|x|^2}{2}$ arising from smooth and small initial data. Then, the corresponding spatial density has a self-similar asymptotic profile, in the sense that $$\forall (t,x)\in [0,\infty)\times \R^2_x, \qquad \Big|e^{2t}\rho(f)(t,x)-\int_{\R^2_s}\bar{f}_{\infty}\Big(s,\frac{x}{e^t}\Big)\mathrm{d}s\Big|\lesssim \frac{(1+t)^{7}}{e^{2t}}.$$ Moreover, the spatial density satisfies $$\forall (t,x)\in [0,\infty)\times \R^2_x, \qquad \Big|e^{2t}\rho(f)(t,x)-\int_{\R^2_s}\bar{f}_{\infty}(s,0)\mathrm{d}s\Big|\lesssim (1+|x|)\frac{(1+t)^{7}}{e^{t}}.$$
\end{theorem}

\begin{remark}
\begin{enumerate}[label = (\alph*)]
\item As we commented before, we use the hyperbolicity of the Hamiltonian flow, by making use of the vector fields contained in the stable and unstable invariant distributions of phase space for the linearized system. As a result, we obtain optimal exponential decay in time for the induced spatial density. The rate of exponential decay for the spatial density coincides with the \emph{sum of all positive Lyapunov exponents} of the Hamiltonian flow. Moreover, we obtain a self-similar asymptotic profile $\int_{\R^2_s}\bar{f}_{\infty}(s,\cdot)\mathrm{d}s$ of the spatial density in terms of the scattering state. We observe that $\int_{\R^2_s}\bar{f}(t,s,0)\mathrm{d}s$ is a conservation law along the stable manifold of the origin for the linearized system. As a result, the limit at infinity of the corresponding conserved quantities for the linearized system describes the late time asymptotic behavior of the spatial density. 
\item We observe that the self-similar asymptotic profile $\int_{\R^2_s}\bar{f}_{\infty}(s,\cdot)\mathrm{d}s$ of the spatial density is defined by integrating the scattering state on the stable manifold of the origin. In terms of the initial data, this quantity corresponds to integrating the initial distribution function on the trapped set at time $\{t=0\}$. The second two authors have obtained an explicit teleological construction of the trapped set in terms of the evolution in time of the force field (see \cite[Theorem 1.3]{VV23} for further details).
\end{enumerate}
\end{remark}

The decay in time of the spatial density holds due to the concentration in time of the support of the distribution function on the unstable manifold of the origin. Motivated by this fact, we capture the concentration of the support of the distribution function in the unstable manifold with a suitable weak convergence statement. 

\begin{theorem}[Concentration of the distribution in the unstable manifold]
Let $\varphi\in C^{\infty}_{x,v}$ be a compactly supported test function. Then, for every solution $f$ to the Vlasov--Poisson system with the potential $\frac{-|x|^2}{2}$ arising from smooth and small initial data, we have $$\lim_{t\to\infty}e^{2t}\int_{\R^2_s\times\R^2_u}\bar{f}(t,s,u)\bar{\varphi}(s,u)\mathrm{d}s\mathrm{d}u=\int_{\R^2_s} \bar{f}_{\infty}(s,0)\mathrm{d}s \int_{\R^2_s\times \R^2_u} \bar{\varphi}(s,u)\delta_{s=0}(s)\mathrm{d}s\mathrm{d}u.$$ In other words, the distribution function $e^{2t}\bar{f}(t,s,u)$ converges weakly to $(\int \bar{f}_{\infty}(s,0)\mathrm{d}s)\delta_{s=0}(s)$.
\end{theorem}

\begin{remark}
We observe that the mass of the corresponding Dirac measure is explicitly identified, as the mass of the stable manifold of the origin in terms of the scattering state. The mass of the Dirac measure is equal to the mass of the trapped set in terms of the initial distribution function. In the core of the paper, we prove a more general weak convergence statement for $e^{2t}\bar{f}(t,s,u+\bar{u}e^t)$, for a fixed $\bar{u}\in \R^2_u$. We show that $e^{2t}\bar{f}(t,s,u+\bar{u}e^t)$ converges weakly to the Dirac mass $(\int \bar{f}_{\infty}(s,\bar{u})\mathrm{d}s)\delta_{s=0}(s)$. We observe that the masses $\int \bar{f}_{\infty}(s,\bar{u})\mathrm{d}s$ of these Dirac measures are explicitly identified, as the masses along the leaves $\{u=\bar{u}\}$ in terms of the scattering state. Note that the mass of the stable leaves $\{u=\bar{u}\}$ defines the self-similar asymptotic profile of the spatial density.
\end{remark}

For every sufficiently regular solution $f$ to the Vlasov--Poisson system with the potential $\frac{-|x|^2}{2}$, we consider the \emph{Hamiltonian energy of the system}, given by 
\begin{equation}
\mathcal{H}[f]:=\dfrac{1}{2}\int_{\R_x^2\times \R^2_v}(|v|^2-|x|^2)f(t,x,v)\mathrm{d}x\mathrm{d}v-\dfrac{\mu}{2}\int_{\R^2_x}|\nabla_x\phi|^2(t,x)\mathrm{d}x.
\end{equation}
The Hamiltonian energy of a solution to the Vlasov--Poisson system with the potential $\frac{-|x|^2}{2}$ is \emph{conserved in time}. This quantity has a key role in the Hamiltonian structure of this PDE system. We conclude this subsection with an explicit characterization of the Hamiltonian energy of the system in terms of the scattering state, in the class of small data solutions studied in this paper.

\begin{theorem}
Let $f$ be a solution to the Vlasov--Poisson system with the potential $\frac{-|x|^2}{2}$ arising from small data. Then, the Hamiltonian energy of the system is equal to the asymptotic one induced by the scattering state. In other words, we have $$\mathcal{H}[f(t)]=\mathcal{H}[f_{\infty}]:=\dfrac{1}{2}\int_{\R_x^2\times \R^2_v}(|v|^2-|x|^2)f_{\infty}(x,v)\mathrm{d}x\mathrm{d}v-\dfrac{\mu}{2}\int_{\R^2_x}|\nabla_u\phi_{\mathrm{asymp}} |^2(u)\mathrm{d}u.$$
\end{theorem}

\begin{remark}
\begin{enumerate}[label = (\alph*)]
\item The asymptotic Hamiltonian energy $\mathcal{H}[f_{\infty}]$ of the system in terms of the scattering state has a contribution coming from the self-similar asymptotic profile $\nabla_u\phi_{\mathrm{asymp}}[Q_{\infty}]$. The contribution of this asymptotic profile does not appear in the corresponding formula for the asymptotic Hamiltonian energy $\mathcal{H}[f_{\infty}]$ in terms of the scattering state for the standard Vlasov--Poisson system on $\R^3_x\times \R^3_v$. 
\item Later in the paper, we also obtain an explicit characterization of the total mass of the system $\|f\|_{L^1_{x,v}}$ in terms of the scattering state, in the class of small data solutions studied in this paper. We show that $\|f(t)\|_{L^1_{x,v}}$ is equal to the total mass $\|f_{\infty}\|_{L^1_{x,v}}$ induced by the scattering state $f_{\infty}$.
\end{enumerate}
\end{remark}

\subsection{Outline of the paper}
The article is structured as follows.
\begin{itemize}
    \item \textbf{Section \ref{preliminaries_unstable_potential}.} We study the linearization with respect to the vacuum solution of the Vlasov--Poisson system with the potential $\frac{-|x|^2}{2}$. We introduce the weights and vector fields used to define the norm considered in Theorem 1. We conclude with some basic lemmata for the commuted equations.
\item \textbf{Section \ref{section_main_results}.} We state detailed statements of the main results of the paper.
\item \textbf{Section \ref{section_small_data_global_existence}.} We set up the bootstrap assumptions and discuss their consequences. Later, we prove that weighted $L^{\infty}_{x,v}$ norms of the distribution function grow at most polynomially in time. We improve the bootstrap assumptions on velocity averages, and we conclude the small data global existence part of the paper.  
\item \textbf{Section \ref{section_modified_scattering}.}
We refine the decay estimates, by proving that the spatial density and the force field have self-similar asymptotic profiles. This profiles allow us to define the modified trajectories along which the distribution function converges. We prove small data modified scattering for the distribution function.     
\item \textbf{Section \ref{section_asymp_prop_scatt_state}.} 
We obtain the late-time asymptotic behavior of the spatial density. We prove that the distribution function (up to normalization) converges weakly to a Dirac mass on the unstable manifold of the origin. We also capture the hyperbolicity of the system with a more general weak convergence statement. Finally, we relate the Hamiltonian energy of the system to the corresponding asymptotic Hamiltonian energy of the scattering state. 
\end{itemize}

\begin{ack}
LB conducted this work within the France 2030 framework programme, the Centre Henri Lebesgue ANR-11-LABX-0020-01. AVR received funding from the grant FONDECYT Iniciaci\'on 11220409. RVR would like to thank Jacques Smulevici for many insightful discussions. RVR received funding from the European Union’s Horizon 2020 research and innovation programme under the Marie Skłodowska-Curie grant 101034255. 
\end{ack}

\section{Preliminaries}\label{preliminaries_unstable_potential}

In this section, we introduce the class of commuting vector fields used to study dispersion estimates for the Vlasov--Poisson system with the external potential $\frac{-|x|^2}{2}$. This is motivated by the dynamics defined by the characteristics to the linearized system. Furthermore, we prove some useful lemmata which are going to be applied in the following sections.

\subsection{The Vlasov equation with the external potential \texorpdfstring{$\frac{-|x|^2}{2}$}{x2}}\label{subsection_vlasov_linear_flow}

In this subsection, we study the dynamics of the linearization of the non-linear Vlasov--Poisson system with the trapping potential \eqref{vlasov_poisson_unstable_trapping_potential_paper} with respect to its vacuum solution, which is given by the \emph{Vlasov equation with the external potential $\frac{-|x|^2}{2}$} taking the form
\begin{equation}\label{vlasov_linear_flow}
\begin{cases}
\partial_t f+v\cdot\nabla_xf+x\cdot \nabla_vf=0,\\ 
f(t=0,x,v)=f_0(x,v),
\end{cases}
\end{equation}
where $f_0:\R^{2}_x\times \R^{2}_v\to [0,\infty)$ is a sufficiently regular initial data. We emphasize that this linear Vlasov equation is a transport equation along the Hamiltonian flow given by
\begin{equation}\label{linear_hyperbolic_ode_system}
    \dfrac{dx^i}{dt}=v^i,\qquad \dfrac{\mathrm{d}v^i}{dt}=x^i,
\end{equation}
defined by the Hamiltonian system $(\R^{2}_x\times \R^{2}_v, H)$ in terms of the Hamiltonian  $$H(x,v):=\dfrac{1}{2}(v^1)^2+\dfrac{1}{2}(v^2)^2-\dfrac{1}{2}(x^1)^2-\dfrac{1}{2}(x^2)^2.$$ The Hamiltonian system $(\R^{2}_x\times \R^{2}_v,H)$ is \emph{completely integrable in the sense of Liouville} due to the two independent conserved quantities in involution $$H^i(x,v):=\dfrac{1}{2}(v^i)^2-\dfrac{1}{2}(x^i)^2,$$ where $i\in \{1,2\}$, whose sum yields the total Hamiltonian $H$. The flow map of the Vlasov system is given by the explicit formula  $$(X_{\L}(t,x,v),V_{\L}(t,x,v))=(x\cosh t +v\sinh t ,x\sinh t+v\cosh t).$$
In particular, we can write an explicit solution to the Vlasov equation \eqref{vlasov_linear_flow} using the formula $$f(t,x,v)=f_0(X_{\L}(-t,x,v),V_{\L}(-t,x,v)).$$

\begin{lemma}
Let $f_0$ be an initial data for the Vlasov equation \eqref{vlasov_linear_flow}. Then, the corresponding solution $f$ to the Vlasov equation is given by
\begin{equation}
    f(t,x,v)=f_0(x\cosh t-v\sinh t , v \cosh t-x \sinh t).
\end{equation}
\end{lemma}

\subsection{Macroscopic, microscopic, and unstable vector fields}\label{subsection_macro_micro_unst_vectors}

In this subsection, we introduce classes of vector fields contained in the tangent space of phase space used to study the dispersion of small data solutions for the non-linear Vlasov--Poisson system with a trapping potential \eqref{vlasov_poisson_unstable_trapping_potential_paper} motivated by the explicit dynamics of the linear Vlasov equation \eqref{vlasov_linear_flow}. For this purpose, we introduce the following terminology: we say that a vector field is \emph{macroscopic} if it is contained in the tangent space of $\R^2_x$, and we say that a vector field is \emph{microscopic} if it is contained in the tangent space of $\R^2_x\times \R^2_v$.

Let us consider the following microscopic vector fields which commute with the generator of the linear Vlasov equation \eqref{vlasov_linear_flow} given by $v\cdot \nabla_x+x\cdot \nabla_v$,
\begin{enumerate}[label = (\alph*)]
    \item unstable vector fields $U_i:=e^t(\partial_{x^i}+\partial_{v^i})$,
    \item stable vector fields $S_{i}:=e^{-t}(\partial_{x^i}-\partial_{v^i})$,
    \item scaling vector field $L:= \sum_{i=1}^2x^i\partial_{x^i}+v^i\partial_{v^i}$,
    \item rotation $R_{12}:=x^1\partial_{x^2}-x^2\partial_{x^1}+v^1\partial_{v^2}-v^2\partial_{v^1}$,
\end{enumerate}
and define 
\begin{equation*}
\lambda:=\Big\{ U_i, S_i, L, R_{ij}  \Big\},
\end{equation*}
where $i,j\in \{1,2\}$. The collection of microscopic vector fields $\lambda$ was previously used in \cite{VV23} to set the energy spaces on which the last two authors proved small data global existence for the Vlasov--Poisson system with the potential $\frac{-|x|^2}{2}$. In this article, the stable and unstable vector fields play a more central role compared to the scaling and the rotation vector fields. The vector fields in $\lambda$ commute with the linear Vlasov equation, so the next lemma follows (see also \cite[Lemma 2.2.1]{VV23}).

\begin{lemma}\label{lemma_commutators_Vlasov_external_potential}
Let $f$ be a regular solution of the Vlasov equation with the trapping potential \eqref{vlasov_linear_flow}. Then, $Zf$ is also a solution of this equation for every $Z\in \lambda$. 
\end{lemma}

Let us also consider the following set of vector fields given by
\begin{enumerate}[label = (\alph*)]
    \item unstable vector fields\footnote{The unstable vector fields $\partial_{u^i}$ are \emph{not} the same as the unstable vector fields $U_i$ previously defined. Note the extra exponential weight in the definition of $U_i$.} $\partial_{u^i}$,
    \item unstable scaling vector field $L_u :=  u^1 \partial_{u^1}+u^2 \partial_{u^2}$,
    \item unstable rotation $R_{12,u} := u^1 \partial_{u^2}-u^2 \partial_{u^1}$,
\end{enumerate}
and define 
\begin{equation*}
\lambda_u:=\Big\{ \partial_{u^i}, L_u, R_{ij,u}  \Big\},
\end{equation*}
where $i,j\in \{1,2\}$. The collection of microscopic vector fields $\lambda_{u}$ will be used to study the asymptotic Poisson equation, which describes the asymptotic behavior of the force field of the system. See Subsection \ref{subsection_self_similar_asymp_profile} for more details.

\subsection{Weights preserved along the linear flow}

The set $\mathbf{k}$ of weight functions composed by $$z_{i}^+:=e^{t}\Big(\frac{x^i-v^i}{2}\Big),\qquad z_{i}^-:=e^{-t}\Big(\frac{x^i+v^i}{2}\Big), $$ where $ i\in\{1,2\}$, are conserved along the characteristics $(t,x,v)\mapsto (X_{\L}(t,x,v),V_{\L}(t,x,v))$ of the linear Vlasov equation with the potential $\frac{-|x|^2}{2}$. As a result, the weight functions are solutions to the linear Vlasov equation, in other words, $\forall z\in \mathbf{k}$ we have $\T_0(z)=0$. If $\T_0(g)=0,$ then the same property is satisfied by $zg$, so weighted Sobolev norms of $g$ are conserved for solutions to the linear Vlasov equation. In our nonlinear setting these norms will grow polynomially in time and will then provide useful decay properties for the Vlasov field. For convenience, we define $$\mathbf{z}:=\Big(1+\sum_{z\in \mathbf{k}}z^2\Big)^{\frac{1}{2}},$$ which by construction satisfies $\T_0(\mathbf{z})=0$, and $\mathbf{z}\geq 1$.

\begin{lemma}
Let $Z\in \lambda$, $z\in \mathbf{k}$, and $a\in \N$. Then, we have either $Z(z)\in \mathbf{k}\cup \{0,1\}$ or $- Z(z)\in \mathbf{k}\cup \{0,1\}$. Moreover, we have 
\begin{equation}\label{estimate_weights_wrt_commuting_vector_fields}
|Z(\mathbf{z}^a)|\lesssim |a| \mathbf{z}^a. 
\end{equation}
\end{lemma}

\begin{proof}
If $Z=U_i$, then, we have $$U_i(z_{j}^+)=0 ,\qquad U_i(z_{j}^-)=\delta_{ij},$$ for every $i,j\in \{1,2\}$. If $Z=S_i$, then, we have $$S_i(z_{j}^+)=\delta_{ij} ,\qquad S_i(z_{j}^-)=0,$$ for every $i,j\in \{1,2\}$. If $Z=R_{12}$, then, we have $$R_{12}(z_{1}^+)=- z_{2}^+ ,\quad R_{12}(z_{2}^+)=z_{1}^+ ,\quad R_{12}(z_{1}^-)=- z_{2}^-,\quad R_{12}(z_{2}^-)=z_{1}^-.$$ If $Z=L$, then, we have $$L(z_{j}^+)=z_{j}^+ ,\qquad L(z_{j}^-)=z_{j}^-, $$ for every $j\in \{1,2\}$. The estimate \eqref{estimate_weights_wrt_commuting_vector_fields} follows directly by using the previous identities.
\end{proof}

Motivated by the fact that any regular solution to the linear Vlasov equation $\T_0(h) =0$ is constant along the flow lines, that is $h(t, X_{\L}(t,x,v),V_{\L}(t,x,v))=h(0,x,v)$, it will sometimes be useful to work with $g(t,x,v):= f(t,X_{\L}(t,x,v),V_{\L}(t,x,v))$, in particular when studying the asymptotic properties of $\rho(f)(t,x)$ and its derivatives. The following result suggests that $g$ enjoys strong decay and that none of its derivatives grow exponentially in time.

\begin{lemma}\label{lem_properties_linear_change_variable_trivial}
Let $f:[0,\infty)\times \R^2_x\times \R^2_v\to[0,\infty)$ be a sufficiently regular distribution function and $g(t,x,v):=f(t, X_{\L}(t,x,v), V_{\L}(t,x,v))$. Then, we have 
$$\langle x,v \rangle^a |g(t,x,v)|\lesssim |\mathbf{z}^a f|(t,X_{\L}(t,x,v),V_{\L}(t,x,v)),$$
and,
$$ |(\nabla_x\pm \nabla_v) g| (t,x,v) \leq \sum_{Z\in \lambda} |Zf|(t,X_{\L}(t,x,v), V_{\L}(t,x,v)).$$
\end{lemma}

\begin{proof} Since weights in $\mathbf{k}$ are preserved by the linear Vlasov equation,
$$\mathbf{z}^2(t,X_{\L}(t,x,v), V_{\L}(t,x,v))=\mathbf{z}^2(0,x,v)=1+\frac{1}{4}(|x-v|^2+|x+v|^2)=1+\frac{1}{2}(|x|^2+|v|^2),$$ and thus the first inequality holds. The second inequality follows by noting that $$(\partial_{x^i}\pm\partial_{v^i})g (t,x,v)=e^{\pm t} (\partial_{x^i} \pm \partial_{v^i})f(t,X_{\L}(t,x,v),V_{\L}(t,x,v)).$$
\end{proof}

\begin{remark}There is an explicit correspondence between the set of stable and unstable vector fields, and the weights in $\mathbf{k}$. More precisely, we have $U_i=\{z_i^+,\cdot\},$ and $S_i=\{z_i^-,\cdot\}$, where  $\{\cdot,\cdot\}$ is the Poisson bracket of the standard symplectic structure on $\R_x^n\times \R_v^n$.
\end{remark}

\subsection{Macroscopic, microscopic, and unstable differential operators}

Let $(Z^i)_i$ be an arbitrary ordering of the microscopic vector fields contained in $\lambda$. In the following, we use a multi-index notation for the microscopic differential operators of order $|\alpha|$ given by the composition $$Z^{\alpha}:=Z^{\alpha_1} Z^{\alpha_2}\dots Z^{\alpha_{n}},$$ for every multi-index $\alpha\in \N^{n}$. We denote by $\lambda^{|\alpha|}$ the family of microscopic differential operators obtained as a composition of $|\alpha|$ vector fields in $\lambda$. 

Furthermore, we can uniquely associate a macroscopic differential operator to any microscopic differential operator $Z^{\alpha}\in \lambda^{|\alpha|}$, by replacing every microscopic vector field $Z$ by the corresponding macroscopic vector field $Z_x$, so that
\begin{equation*}
   U_{i,x}=e^t\partial_{x^i},\qquad S_{i,x}=e^{-t}\partial_{x^i}, \qquad 
   L_x=x^1\partial_{x^1}+x^2 \partial_{x^2},\qquad R_{12,x}=x^1\partial_{x^2}-x^2\partial_{x^1}.
\end{equation*}
 By a small abuse of notation, we denote also by $Z^{\alpha}$ to the associated macroscopic differential operator to an arbitrary microscopic differential operator $Z^{\alpha}$. We denote by $\Lambda^{|\alpha|}$ to the family of macroscopic differential operators of order $|\alpha|$ obtained as a composition of $|\alpha|$ vector fields in $\Lambda$. 

Let us now consider a microscopic vector field $Z^{\alpha}$ \emph{without} stable vector fields. In this case, we can uniquely associate an unstable differential operator to any microscopic differential operator $Z^{\alpha}\in \lambda^{|\alpha|}$ by replacing every microscopic vector field $Z$ by the corresponding unstable vector field $Z_u$. By a small abuse of notation, we denote by $Z_u^{\alpha}$ the associated unstable differential operator to an arbitrary microscopic differential operator $Z^{\alpha}$. We denote by $\lambda_u^{|\alpha|}$ to the family of unstable differential operators of order $|\alpha|$ obtained as a composition of $|\alpha|$ unstable vector fields in $\lambda_u$. 

Finally, we denote by $\partial_x^{\alpha}$ a standard macroscopic differential operator $$\partial_x^{\alpha}:=\partial^{\alpha_1}_{x^{1}}\partial^{\alpha_2}_{x^2},$$ for every multi-index $\alpha\in \N\times\N$.

The following two results can be found in \cite[Lemma 2.3.1, Lemma 2.3.3]{VV23}.

\begin{lemma}\label{lemma_commuting_in_lambda}
Let $\Omega\in \{\lambda, \Lambda\}$. Let $\alpha$ and $\beta$ be two multi-indices. Then, the commutator between $Z^{\alpha}\in \Omega^{|\alpha|}$ and $Z^{\beta}\in\Omega^{|\beta|}$ is given by $$[Z^{\alpha}, Z^{\beta}]=\sum_{|\gamma|\le |\alpha|+|\beta|-1}\sum_{Z^\gamma\in \Omega^{|\gamma|}}C^{\alpha \beta}_{\gamma}Z^{\gamma},$$
for some constant coefficients $C^{\alpha \beta}_{\gamma}$.
\end{lemma}

\begin{lemma}\label{lemma_linear_vlasov_weight_vector_fields}
For any multi-index $\alpha$, we have 
\begin{equation}\label{identity_macroscopic_giving_decay}
    (e^t+|x|)^{\alpha}\partial_x^{\alpha}=\sum_{|\beta|\leq |\alpha|}\sum_{Z^{\beta}\in \Lambda^{|\beta|}} C_{\beta}Z^{\beta},
\end{equation}
for some uniformly bounded functions $C_{\beta}$.
\end{lemma}

In the later result we have used the following key identity (see  \cite[Lemma 2.5]{Sm16}),
$$|x|^2 \partial_{x^j}=\sum_{i=1}^nx^iR_{ij,x} + x^j L_x.$$

We conclude this subsection by relating the macroscopic and microscopic differential operators $Z^\alpha$ in terms of their action on $Z^\alpha \rho(f)$ and $\rho(Z^\alpha f)$.

\begin{lemma}\label{lemma_connection_microscopic_macroscopic_vector_fields}
Let $f$ be a sufficiently regular distribution function and let $\alpha$ be a multi-index. Then, there exist constant coefficients $C^{\alpha}_{\beta}$ such that
\begin{equation}
    Z^{\alpha}\rho(f)=\rho(Z^{\alpha}f)+\sum_{|\beta|\leq |\alpha|-1}C^{\alpha}_{ \beta}\rho(Z^{\beta}f),
\end{equation}
where the vector fields in the left hand side are macroscopic, whereas the ones in the right hand side are microscopic.
\end{lemma}

\subsection{The commuted equations}
Let us denote the non-linear transport operator applied to the distribution function in the Vlasov--Poisson system with the external potential $\frac{-|x|^2}{2}$ by $$\T_\phi :=\partial_t +v\cdot\nabla_x +x\cdot\nabla_v -\mu\nabla_x\phi\cdot \nabla_v ,$$ where the force field $\nabla_x\phi$ is defined through the Poisson equation $\Delta_x \phi=\rho(f)$. In order to identify hierarchies in the commuted equations, we introduce the following notation. 
\begin{definition}\label{Defnumberunst}
Let $\alpha$ be a multi-index and $Z^\alpha \in \lambda^{|\alpha|}$ or $Z^\alpha \in \Lambda^{|\alpha|}$. We denote by $\alpha_u$ (respectively $\alpha_s$) the number of unstable (respectively stable) vector fields composing $Z^\alpha$. Then, $|\alpha|=\alpha_u+\alpha_s$ and, for instance, if $Z^\alpha = U_1 S_2 LR_{12}$, we have $\alpha_u =3$ and $\alpha_s =1$.
\end{definition}
By straightforward computations, one obtains the first order commutation formula.
\begin{lemma}\label{LemComfirstorder}
Let $Z \in \lambda$ and define $c_Z:=-2$ if $Z=L$ and $c_Z:=0$ otherwise. Then,
$$ [\T_\phi,Z]=\mu \nabla_x \big( Z \phi+c_Z \phi \big) \cdot \nabla_v .$$
\end{lemma}
Iterating the above, we obtain the higher order case.
\begin{lemma}\label{lemma_commuted_nonliner_Vlasov}
There exist constant coefficients $C^{\alpha}_{\beta \gamma} \in \mathbb{Z}$ such that \begin{equation}\label{eq_comm}
    [\T_{\phi},Z^{\alpha}]=\sum_{|\beta|\leq |\alpha|-1} \, \sum_{|\gamma|+|\beta|\leq |\alpha|} C^{\alpha}_{\beta \gamma}\nabla_x Z^{\gamma}\phi\cdot  \nabla_v Z^{\beta},
\end{equation}
where the vector fields $Z^{\alpha}\in\lambda^{|\alpha|}$, $Z^{\gamma}\in \Lambda^{|\gamma|}$, and $Z^{\beta}\in \lambda^{|\beta|}$. Moreover, we have that either $\beta_u < \alpha_u$, or $\beta_u=\alpha_u$ and $\gamma_s \geq 1$.
\end{lemma}
As obtained in \cite[Lemma 2.4.2]{VV23}, we have the next result for the Poisson equation.
\begin{lemma}\label{lemma_commuted_poisson_equation}
Let $f$ be a sufficiently regular distribution function, and let $\phi$ be the
solution to the Poisson equation $\Delta_x \phi = \rho(f)$. Then, for any multi-index $\alpha$ the function $Z^{\alpha}\phi$ satisfies the equation 
\begin{equation}\label{identity_commuted_poisson_eqn}
\Delta_x Z^{\alpha}\phi=\sum_{|\beta|\leq |\alpha|}C^{\alpha}_{\beta}\rho( Z^{\beta} f),
\end{equation} for some constant coefficients $C^{\alpha}_{\beta}$.
\end{lemma}
\begin{remark}
If $Z^\alpha$ does not contain the scaling vector field, then $\Delta_x Z^{\alpha}\phi=\rho( Z^{\beta} f)$.
\end{remark}

\subsection{Conservation laws}

For every sufficiently regular solution $f$ to the Vlasov--Poisson system with the potential $\frac{-|x|^2}{2}$, we define \emph{the total mass of the system}, given by 
\begin{equation}
\|f\|_{L^1_{x,v}}:=\int_{\R_x^2\times \R^2_v}f(t,x,v)\mathrm{d}x\mathrm{d}v,
\end{equation}
and \emph{the Hamiltonian energy of the system}, given by 
\begin{equation}\label{definition_hamiltonian}
\mathcal{H}[f]:=\dfrac{1}{2}\int_{\R_x^2\times \R^2_v}(|v|^2-|x|^2)f(t,x,v)\mathrm{d}x\mathrm{d}v-\dfrac{\mu}{2}\int_{\R^2_x}|\nabla_x\phi|^2(t,x)\mathrm{d}x.
\end{equation}
By standard arguments in collisionless systems, the following proposition holds.

\begin{proposition}
Let $f$ be a regular classical solution to the Vlasov--Poisson system with the potential $\frac{-|x|^2}{2}$. Then, the total mass and the Hamiltonian energy of the system, are both conserved in time. In other words, for every $t\in [0,\infty)$ we have $$\|f(t)\|_{L^1_{x,v}}=\|f_0\|_{L^1_{x,v}},\qquad \mathcal{H}[f(t)]=\mathcal{H}[f_0].$$
\end{proposition}

The Hamiltonian energy \eqref{definition_hamiltonian} is a central quantity in the Hamiltonian structure of the Vlasov--Poisson system with the potential $\frac{-|x|^2}{2}$.

\section{The main results}\label{section_main_results}

We recall the hyperbolic coordinate system $(s,u)$ in phase space, defined by $$s^i:=\frac{x^i-v^i}{2},\qquad u^i:=\frac{x^i+v^i}{2},$$ which is more suitable to capture the hyperbolicity of the linearized system. We observe that $x^i=s^i+u^i,$ and $v^i=u^i-s^i.$ The coordinate system $(s,u)$ induces the following vector fields in phase space $$\partial_{s^i}=\partial_{x^i}-\partial_{v^i},\qquad \partial_{u^i}=\partial_{x^i}+\partial_{v^i}.$$ We observe that $\partial_{x^i}=\frac{1}{2}(\partial_{s^i}+\partial_{u^i})$, and $\partial_{v^i}=\frac{1}{2}(\partial_{u^i}-\partial_{s^i}).$ We write the distribution function in $(s,u)$ coordinates by $$\bar{f}(t,s,u):=f(t,x,v).$$ Furthermore, we write the derivatives of the distribution function by $Z^{\alpha}\bar{f}=\overline{Z^{\alpha}f}$, where we abuse of notation by writing $Z^{\alpha}$ in terms of the hyperbolic vector fields $\{\partial_{u^i},\partial_{s^i}\}$. We will frequently use these identifications without further references.

\subsection{Precise statements of the main results}

We are now ready to provide a full and detailed version of the main results of the article. First, we prove the global existence of small data solutions to the Vlasov--Poisson system with external potential $\frac{-|x|^2}{2}$. We further provide pointwise estimates for derivatives of the distribution function, and exponential decay in time of the force field. These properties will be needed later to investigate the scattering properties of these small data solutions.

\begin{theorem}\label{thm_detailed_global_existence}   
Let $N \geq 2$, and $f_0$ be an initial data of class $C^N$ for the Vlasov--Poisson system with the potential $\frac{-|x|^2}{2}$. Consider further $\epsilon >0$, a constants $M\in\N$, and assume that 
\begin{equation}\label{assumption_small_data_global_exist_thm}
\sum_{|\beta|+|\kappa|\leq N}\sup_{(s,u)\in \R^2_s\times \R^2_u}\langle s \rangle^{M+|\beta|} \langle u \rangle^{M+|\kappa|}|\partial_s^{\beta}\partial_u^{\kappa} \bar{f}_0|(s,u)\leq \e.
\end{equation}
If $M \geq 6$, there exists $0 <\sigma \leq 1/4$ and $\e_0>0$ such that if $\e\leq\e_0$, then the unique solution $f$ arising from this initial data is global in time. Moreover, the following properties hold.

\begin{enumerate}[label = (\alph*)]
\item \label{estimate_force_field_global_exist} The force field and its derivatives $\nabla_x Z^{\gamma}\phi$ decay exponentially in time. For every $(t,x)\in [0,\infty)\times \R_x^2$, we have \begin{align*}
\forall |\gamma|\leq N-1, \qquad &|e^{t}\nabla_x Z^{\gamma}\phi|(t,x)\lesssim \e , \\
\forall |\gamma| = N, \qquad &|e^{t}\nabla_x Z^{\gamma}\phi|(t,x)\lesssim \e \, e^{\sigma t}.
\end{align*}
\item \label{point_estimate_statement_global_exist} The following $L^{\infty}_{x,v}$ estimates hold for the derivatives of the distribution function. For every $(t,x,v)\in [0,\infty)\times \R_x^2\times \R^2_v$, we have
\begin{align*}
\forall |\beta|\leq N-1, \qquad &|\z^{M}Z^{\beta}f|(t,x,v)\lesssim \e (1+t)^{M+N-1}, \\
\forall |\beta|=N, \qquad &|\z^{M}Z^{\beta}f|(t,x,v)\lesssim \e (1+t)^{M}e^{\sigma t}, \\
\forall |\kappa|\leq N-1, \qquad &|e^{-t|\kappa|}(\partial_x-\partial_v)^{\kappa}f|(t,x,v)\lesssim \e.
\end{align*}
\item The spatial density and its derivatives $\rho(Z^{\beta} f)$ decay exponentially in time. For every $(t,x)\in [0,\infty)\times \R_x^2$, we have
\begin{align*}
\forall |\beta|< N, \qquad &e^{2t}|\rho(Z^{\beta}f)|(t,x)\lesssim \e , \\
\forall |\beta| = N, \qquad &e^{2t}|\rho(Z^{\beta}f)|(t,x)\lesssim \e \, e^{\sigma t}.
\end{align*}
 \end{enumerate}
\end{theorem}

We now proceed to state the main scattering result of the article. Given a function $Q:\R^2_u\to \R$. In the following, we denote the unique solution $\phi:\R^2_u\to \R$ to the Poisson equation $\Delta_u \phi=Q$ by $\phi_{\mathrm{asymp}}[Q]$.

\begin{theorem}\label{thm_detail_modified_scattering}
Let $f$ be a smooth solution to the Vlasov--Poisson system with the potential $\frac{-|x|^2}{2}$ arising from initial data satisfying the assumptions of Theorem \ref{thm_detailed_global_existence}. Then, the following properties are satisfied.
\begin{enumerate}[label=(\alph*)]
\item \label{first_prop_mod_scatt_main} The normalized stable averages of $Z^\beta f$ along $\{u=u_0\}$ converge to a function $Q^\beta_{\infty}\in  L^{\infty}(\R^2_u)$ of class $C^{N-1-|\beta|}$. For every $|\beta|\leq N-1$, we have $$\forall t\in [0,\infty), \qquad \Big|\langle u \rangle^{M-3}\Big(e^{2t}\int_{\R^2_s}Z^\beta \bar{f}(t,s,e^tu)\mathrm{d}s -Q^\beta_{\infty}(u) \Big)\Big|\lesssim \e\dfrac{(1+t)^{N+2}}{e^{2t}},$$ where $Q^{\beta}_{\infty}$ can be computed explicitly in terms of $\partial^{\kappa}_uQ_{\infty}$ for $|\kappa|\leq |\beta|$.
\item \label{second_prop_mod_scatt_main} The spatial density $\rho(Z^{\beta}f)$ has a self-similar asymptotic profile. For every $|\beta| \leq N-1$, we have $$\forall t \in [0,\infty),\qquad \Big| e^{2t}\int_{\R^2_v}Z^{\beta}f(t,x,v)\mathrm{d}v-Q^{\beta}_{\infty}\Big(\dfrac{x}{e^t}\Big)\Big|\lesssim \e \dfrac{(1+t)^{N+5}}{e^{2t}}.$$ 
\item \label{third_prop_mod_scatt_main} The force field and its derivatives $\nabla_x Z^{\gamma}\phi$ have a self-similar asymptotic profile. For every $|\gamma| \leq N-1$, we have $$  \forall t\in [0,\infty),\qquad \Big| e^t\nabla_x Z^{\gamma}\phi(t,X_{\L}(t))-\nabla_u\phi_{\mathrm{asymp}}[Z^{\gamma}_uQ_{\infty}](u)\Big|\lesssim \e \, \langle e^{-t} s \rangle \,  \dfrac{(1+t)^{N+5}}{e^{ t}}.$$ 
\item If $N \geq 3$, the distribution function $\bar{f}$ has modified scattering to a distribution $\bar{f}_{\infty}\in L^{\infty}_{s,u}$ of class $C^{N-3}$. For any $|\kappa|+|\beta|\leq N-3$, we have $$\color{white} \square \quad \qquad \color{black} \forall t\in [0,\infty),\qquad \Big|\langle s\rangle^{M-1}\langle u \rangle^M \big(\partial_u^{\kappa}\partial_s^{\beta}\bar{f}(t,S_{\C}, u)-\partial_u^{\kappa}\partial_s^{\beta}\bar{f}_{\infty}(s,u) \big)\Big| \lesssim \e\dfrac{(1+t)^{3N+M+1}}{e^{ t}},$$ where the component $S_{\C}$ of the modified stable characteristics is defined as 
\begin{align*} 
S_{\C}(t,s,u):=e^{-t}\Big(s+\frac{1}{2} \mu t\nabla_u\phi_{\mathrm{asymp}}[Q_{\infty}](u)\Big). 
\end{align*}
\item The asymptotic modified vector fields of the unstable vector fields $U_i$, the rotation $R_{ij}$, and the scaling $L$, given respectively by 
\begin{align*}
U_i^{\mathrm{mod}}&:=U_i+\frac{1}{2}\mu t\sum_{k=1}^2\partial_{u^k} \phi_{\mathrm{asymp}} \big[ \partial_{u^i}Q_\infty \big]\left( \frac{u}{e^t} \right)S_k,\\
R_{ij}^{\mathrm{mod}}&:= R_{ij}+\frac{1}{2}\mu t\sum_{k=1}^2\partial_{u^k} \phi_{\mathrm{asymp}} \big[ R_{ij,u}Q_\infty \big]\left( \frac{u}{e^t} \right)S_k,\\
L^{\mathrm{mod}}&:=L+\frac{1}{2}\mu t\sum_{k=1}^2\partial_{u^k} \phi_{\mathrm{asymp}} \big[ L_uQ_\infty \big]\left( \frac{u}{e^t} \right)S_k-2 \partial_{u^k} \phi_{\mathrm{asymp}} \big[ Q_\infty \big]\left( \frac{u}{e^t} \right)S_k,
\end{align*}
verify the improved estimates $$\|U_i^{\mathrm{mod}}f\|_{L^{\infty}_{x,v}}\lesssim \e,\quad \|R_{ij}^{\mathrm{mod}}f\|_{L^{\infty}_{x,v}}\lesssim \e,\quad \|L^{\mathrm{mod}}f\|_{L^{\infty}_{x,v}}\lesssim\e.$$
\end{enumerate}
\end{theorem}
\begin{remark}
The statements \ref{first_prop_mod_scatt_main}, \ref{second_prop_mod_scatt_main} and \ref{third_prop_mod_scatt_main} hold true as well for any $|\beta|=N-1$, but with the weaker rate of convergence $\langle t \rangle^{N+6}e^{-(2-\sigma)t}$. Similarly, $\bar{f}_\infty$ is in fact of class $C^{N-2}$.
\end{remark}
\begin{remark}
We consider asymptotic modified vector fields to prove regularity of the scattering state with respect to the unstable variable $u$. The modification to the commuting vector fields for the linearized system grows linearly in time. Compare the asymptotic modified vector fields $U_i^{\mathrm{mod}}$, $R_{ij}^{\mathrm{mod}}$, and $L^{\mathrm{mod}}$, with the modified vector fields previously considered in \cite[Subsection 5.1]{VV23} to show small data global existence. The corrections in the asymptotic modified vector fields $U_i^{\mathrm{mod}}$, $R_{ij}^{\mathrm{mod}}$, and $L^{\mathrm{mod}}$, are identified dynamically in terms of the asymptotic behavior of the force field and its derivatives.
\end{remark}

\begin{theorem}\label{thm_asymp_spatial_density_complete}
Let $N\geq 3$ and $M\geq 6$. Let $f\in C^N$ be a solution to the Vlasov--Poisson system with the potential $\frac{-|x|^2}{2}$ arising from small data. Then, the corresponding spatial density has a self-similar asymptotic profile, in the sense that $$\forall (t,x)\in [0,\infty)\times \R^2_x, \qquad \Big|e^{2t}\rho(f)(t,x)-\int_{\R^2_s}\bar{f}_{\infty}\Big(s,\frac{x}{e^t}\Big)\mathrm{d}s\Big|\lesssim \e\frac{(1+t)^{7}}{e^{2t}}.$$ Moreover, the spatial density satisfies $$\forall (t,x)\in [0,\infty)\times \R^2_x, \qquad \Big|e^{2t}\rho(f)(t,x)-\int_{\R^2_s}\bar{f}_{\infty}(s,0)\mathrm{d}s\Big|\lesssim \e(1+|x|)\frac{(1+t)^{7}}{e^{t}}.$$ 
\end{theorem}

Next, we capture the concentration of the support of the distribution function in the unstable manifold with a suitable weak convergence statement. We also capture the hyperbolicity of the non-linear system with a more general weak convergence statement for $\bar{f}(t,s,u+\bar{u}e^t)$ for a fixed $\bar{u}\in\R^2_u$.

\begin{theorem}\label{thm_weak_convergence_prop_complete}
Let $\varphi\in C^{\infty}_{s,u}$ be a compactly supported test function. Let $\bar{u}\in\R^2_u$. Let $f$ be a small data solution to the Vlasov--Poisson system with the potential $\frac{-|x|^2}{2}$. Then, the distribution $e^{2t}\bar{f}(t,s,u)$ converges weakly to $(\int \bar{f}_{\infty}(s,0)\mathrm{d}s)\delta_{s=0}(s)$. In other words, we have $$\lim_{t\to\infty}e^{2t}\int_{\R^2_s\times\R^2_u}\bar{f}(t,s,u)\bar{\varphi}(s,u)\mathrm{d}s\mathrm{d}u=\int_{\R^2_s} \bar{f}_{\infty}(s,0)\mathrm{d}s \int_{\R^2_s\times \R^2_u} \bar{\varphi}(s,u)\delta_{s=0}(s)\mathrm{d}s\mathrm{d}u.$$ Moreover, the distribution $e^{2t}\bar{f}(t,s,u+\bar{u}e^t)$ converges weakly to $(\int \bar{f}_{\infty}(s,\bar{u})\mathrm{d}s)\delta_{s=0}(s)$. In other words, we have $$\lim_{t\to\infty}e^{2t}\int_{\R^2_s\times \R^2_u}\bar{f}(t,s,u+\bar{u}e^t)\bar{\varphi}(s,u)\mathrm{d}s\mathrm{d}u=\int_{\R^2_s} \bar{f}_{\infty}(s,\bar{u})\mathrm{d}s \int_{\R^2_s\times \R^2_u} \bar{\varphi}(s,u)\delta_{s=0}(s)\mathrm{d}s\mathrm{d}u.$$ 
\end{theorem}

Finally, we show an explicit characterization of the total mass and the Hamiltonian energy of the system in terms of the scattering state, in the class of small data solutions studied in this paper.

\begin{theorem}\label{thm_conservation_laws_complete}
Let $f$ be a solution to the Vlasov--Poisson system with the potential $\frac{-|x|^2}{2}$ arising from small data. Then, the total mass and the Hamiltonian energy of the system are equal to the asymptotic ones induced by the scattering state. In other words, we have $$\|f\|_{L^1_{x,v}}=\|f_{\infty}\|_{L^1_{x,v}},$$ and $$\mathcal{H}[f(t)]=\mathcal{H}[f_{\infty}]:=\dfrac{1}{2}\int_{\R_x^2\times \R^2_v}(|v|^2-|x|^2)f_{\infty}\mathrm{d}x\mathrm{d}v-\dfrac{\mu}{2}\int_{\R^2_x}|\nabla_u\phi_{\mathrm{asymp}}[Q_{\infty}] |^2\mathrm{d}u.$$
\end{theorem}

\section{Global existence of small data solutions}\label{section_small_data_global_existence}
In this section we prove Theorem \ref{thm_detailed_global_existence}, which states the global existence of small data solutions with respect to the weighted $L_{x,v}^\infty$ norm in \eqref{assumption_small_data_global_exist_thm}. For this purpose, we will prove decay in time of the spatial density and its derivatives via a bootstrap argument. Parts \ref{point_estimate_statement_global_exist} and \ref{estimate_force_field_global_exist} in Theorem \ref{thm_detailed_global_existence} are proved along with the proof of the bootstrap argument. The estimates obtained in this section are crucial inputs in the proof of the modified scattering theorem obtained in Section \ref{section_modified_scattering}.

\subsection{The bootstrap argument}
Let $N\geq 2$ and $M\geq 6$. Let us consider an initial data $f_0$ satisfying the hypotheses of Theorem \ref{thm_detailed_global_existence}. By a standard local well-posedness argument, there exists a unique maximal solution $f$ to the Vlasov--Poisson system with the potential $\frac{-|x|^2}{2}$ arising from this data. Let $T_{\mathrm{max}}\in (0,\infty]$ be the maximal time such that the solution $f$ to the Vlasov--Poisson system is defined on $[0,T_{\mathrm{max}})$. By continuity, there exists a largest time $T\in [0,T_{\mathrm{max}}]$, and a constant $C_{\mathrm{boot}}>0$ such that the following bootstrap assumption holds:
\begin{enumerate}[label = BA\arabic*]
\item For every $(t,x)\in [0,T)\times \R^2_x$ and every $ |\beta|\leq N-1$, we have $$\Big|\int_{\R^2_v}Z^{\beta}f(t,x,v)\mathrm{d}v\Big| \leq \dfrac{C_{\mathrm{boot}}\e}{(e^{t}+|x|)^{2}}.$$ \label{boot2}
\item For every $(t,x)\in [0,T)\times \R^2_x$ and every $ |\beta| = N$, we have $$\Big|\int_{\R^2_v}Z^{\beta}f(t,x,v)\mathrm{d}v\Big| \leq \frac{ C_{\mathrm{boot}}\epsilon \, e^{\sigma t}}{(e^{t}+|x|)^{2}},$$ where $0 < \sigma \leq 1/4$ is a fixed constant.\label{boot3}
\end{enumerate}
We will improve these estimates when $\e>0$ is small enough, for a constant $C_{\mathrm{boot}}>0$ chosen sufficiently large. \\

\textbf{Structure of the proof of small data global existence}
\begin{enumerate}[label=(\alph*)]
\item First, we prove decay estimates in time for the force field and its derivatives $\nabla_xZ^\gamma\phi$. We consider non-linear modifications of the weights $z_i^+$, which are defined to be preserved by the non-linear Vlasov equation. We prove that these \emph{modified weights} grow at most linearly in time.
\item We prove that for every $|\beta|\leq N-1$, a weighted $L^{\infty}_{x,v}$ norm of $Z^{\beta}f$ grows at most polynomially in time. At the top order $|\beta|=N$, we will merely be able to close the estimates with an $e^{\sigma t}$ growth. Next, we obtain uniform boundedness in time of normalized weighted stable averages of $Z^{\beta}f$ for every $|\beta|\leq N-1$. These estimates allow us to prove exponential decay in time for velocity averages $\rho(Z^{\beta}f)$ for $|\beta|\leq N-1$ and improve the bootstrap assumptions \eqref{boot2}-\eqref{boot3}.
\end{enumerate}

\subsection{Pointwise decay estimates for the force field}
We start slowly with an elementary calculus lemma. 

\begin{lemma}\label{lemma_uniform_integral_bound_kernel_convolution_duan}
There exists a uniform constant $C>0$, such that for every $x\in \R^2$ we have
\begin{equation*}
    \int_{\R_y^2}\dfrac{\mathrm{d}y}{|y|(1+|x+y|)^2} \leq C.
\end{equation*}
\end{lemma}
\begin{proof}
First, we observe that in the region $\{|y|\leq 1\}$, we have
$$ \int_{|y| \leq 1 }\dfrac{\mathrm{d}y}{|y|(1+|x+y|)^2} \leq \int_{r=0}^1 \int_{\theta=0}^{2 \pi} \frac{r \, \mathrm{d} r}{r} \mathrm{d} \theta \leq 2 \pi.$$
We deal with the remaining region $\{|y|> 1\}$, by applying Hölder inequality from where we have
\begin{align*}
\int_{|y| \geq 1 }\dfrac{\mathrm{d}y}{|y|(1+|x+y|)^2} &\leq \bigg| \int_{|y| \geq 1 }\dfrac{\mathrm{d}y}{|y|^3} \bigg|^{\frac{1}{3}} \, \bigg| \int_{|y| \geq 1 }\dfrac{\mathrm{d}y}{(1+|x+y|)^3}   \bigg|^{\frac{2}{3}} \leq 2\int_{\R^2_z} \frac{\mathrm{d} z}{(1+|z|)^3} <+\infty.
\end{align*}
\end{proof}
As a consequence of Lemma \ref{lemma_uniform_integral_bound_kernel_convolution_duan}, we obtain decay in time for the integral term
\begin{equation}\label{remark_uniform_integral_bound_kernel_convolution_duan}
    \int_{\R_y^2}\dfrac{1}{|y|(e^t+|x-y|)^2} \mathrm{d}y=\dfrac{1}{e^{t}}\int_{\R_{y'}^2} \dfrac{1}{|y'|(1+|y'-\frac{x}{e^t}|)^2}\mathrm{d}y'\lesssim \dfrac{1}{e^{t}},
\end{equation}
by using the change of variables $y=e^t y'$. We use the estimate \eqref{remark_uniform_integral_bound_kernel_convolution_duan} to prove decay for the gradient $\nabla_x Z^\gamma\phi$. 

\begin{proposition}\label{proposition_estimate_phi}
For every $|\gamma|\leq N-1$ and every $(t,x)\in [0,T)\times \R^2_x$, we have $$|\nabla_x Z^{\gamma}\phi| (t,x) \lesssim \dfrac{\e}{e^{(1+2\gamma_s)t}}.$$ For the top order derivatives $|\gamma|=N$, there holds
$$|\nabla_x Z^{\gamma}\phi| (t,x) \lesssim \dfrac{\e }{e^{(1-\sigma+2\gamma_s)t}}.$$
\end{proposition}

\begin{proof}
Combining the commuted Poisson equation in Lemma \ref{lemma_commuted_poisson_equation} with the relation between the macroscopic and microscopic vector fields established in Lemma \ref{lemma_connection_microscopic_macroscopic_vector_fields}, we obtain 
$$\Delta_x Z^{\gamma}\phi=\sum_{|\gamma'|\leq |\gamma|}C^{\gamma}_{\gamma'}\rho(Z^{\gamma'}f),$$ for some constants $C^{\gamma}_{\gamma'}>0$. We use the Green function for the Poisson equation in $\R^2$ to write the solution of the commuted Poisson equation as $$Z^{\gamma}\phi(t,x)=\sum_{|\gamma'|\leq |\gamma|}\int_{\R_y^2} CC_{\gamma'}^{\gamma}\log|y|\rho(Z^{\gamma'}f)(t,x-y)\mathrm{d}y,$$ whose gradient can be estimated directly by 
\begin{equation}\label{nabla_phi}
|\nabla_x Z^{\gamma}\phi(t,x)|\lesssim \sum_{|\gamma'|\leq |\gamma|}\int_{\R_y^2} \dfrac{1}{|y|}|\rho(Z^{\gamma'}f)|(t,x-y)\mathrm{d}y.\end{equation} Hence, the solution of the commuted Poisson equation satisfies that for every $|\gamma|\leq N-1$, we have $$|\nabla_x Z^{\gamma}\phi(t,x)|\lesssim \e \int_{\R_y^2} \dfrac{\mathrm{d}y}{|y|(e^t+|x-y|)^2}\lesssim \frac{\e}{e^{t}},$$ 
where we have used the bootstrap assumption \eqref{boot2} and the estimate \eqref{remark_uniform_integral_bound_kernel_convolution_duan}. We get similarly from the bootstrap assumption \eqref{boot3} and \eqref{remark_uniform_integral_bound_kernel_convolution_duan} that
$$\forall \, |\gamma|=N, \qquad  |\nabla_x Z^{\gamma}\phi(t,x)|\lesssim \epsilon \, e^{-(1-\sigma)t}  .$$
 The improved estimate in terms of $\beta_s$ follows directly by rewritting the stable vector fields in $Z^{\gamma}$ in terms of the corresponding unstable vector fields, $S_k=e^{-2t}U_k$.
\end{proof}

\subsection{Key lemma for the $L^{\infty}_{x,v}$ bounds of the distribution}

In this subsection, we prepare the ground to prove $L^{\infty}_{x,v}$ bounds for the distribution function based on the method of characteristics. For this purpose, we prove the following technical lemma.

\begin{lemma}\label{lem_characteristics_integrat_technic_pointwise}
Let $a\in [0,\infty)$ and $g: [0,T)\times \R^2_x\times \R^2_v\to\R$ be a sufficiently regular distribution function such that $$\forall (t,x,v)\in [0,T)\times \R^2_x\times \R^2_v, \qquad |\T_{\phi}(g)|(t,x,v)\leq C_g (1+t)^a,$$ for some constant $C_g>0$. Then, there exists $C>0$ depending only on $a$, such that $$\forall (t,x,v)\in [0,T)\times \R^2_x\times \R^2_v, \qquad |g(t,x,v)|\leq \|g(0)\|_{L^{\infty}_{x,v}}+\frac{CC_g}{a+1} (1+t)^{a+1}.$$ Moreover, if a distribution function $h: [0,T)\times \R^2_x\times \R^2_v\to\R$ verifies, for some constants $C_h>0$ and $\delta>0$, that $$\forall (t,x,v)\in [0,T)\times \R^2_x\times \R^2_v, \qquad |\T_{\phi}(h)|(t,x,v)\leq  C_h e^{-t\delta},$$ then, there exists $C>0$ such that $$\forall (t,x,v)\in [0,T)\times \R^2_x\times \R^2_v, \qquad |h(t,x,v)|\leq C\|h(0)\|_{L^{\infty}_{x,v}}+CC_h\delta^{-1}.$$
\end{lemma}

\begin{proof}
Let $g_0$ and $h_0$ be functions defined as the solutions to 
\begin{alignat*}{2}
\T_{\phi}(g_0)&=(1+t)^a,\qquad &&g_0(0,x,v)=0,\\
\T_{\phi}(h_0)&=e^{-t\delta},\qquad &&h_0(0,x,v)=0.
\end{alignat*}
By Duhamel's formula, for every $[0,T)\times \R^2_x\times\R^2_v$ we have 
\begin{equation}\label{formula_duham_point_estim}
|g|\leq \|g(0)\|_{L^{\infty}_{x,v}}+C_g|g_0|,\qquad |h|\leq \|h(0)\|_{L^{\infty}_{x,v}}+C_h|h_0|.
\end{equation}
Fix a point $(t,x,v)\in [0,\infty)\times\R^2_x\times\R^2_v$. Let us denote by $(X(s),V(s))$ the characteristic flow associated to the transport operator $\T_{\phi}$ such that $$\dfrac{d}{\mathrm{d}s}X(s)=V(s),\qquad \dfrac{d}{\mathrm{d}s}V(s)=X(s)-\mu \nabla_x\phi(s,X(s)),$$ where $X(t)=x$ and $V(t)=v$. Using the method of the characteristics, we have 
\begin{align*}
g_0(t,x,v)&=\int_0^t (1+s)^a \mathrm{d}s= \frac{1}{a+1}(1+t)^{a+1},\\
h_0(t,x,v)&=\int_0^t e^{-s\delta} \mathrm{d}s\leq \frac{1}{\delta}.
\end{align*}
The lemma follows by using these estimates in \eqref{formula_duham_point_estim}.
\end{proof}

\subsection{The modified weights}

Since we expect $e^t(x-v)$ to grow as $t$ along the nonlinear flow, we will rather work with the following modification of this weight.
\begin{definition}
Let $\varphi=(\varphi^1,\varphi^2): [0,T) \times \R^2_x \times \R^2_v \to \R^2$ be the unique solution to
$$\mathbf{T}_\phi (\varphi^i)=- \mathbf{T}_\phi (e^{t}(x^i-v^i)), \qquad \varphi(0,x , v) =0.$$
We define the \emph{modified weight function} $\mathbf{z}_{\mathrm{mod}}$ as
$$\mathbf{z}_{\mathrm{mod}}(t,x,v) := \langle e^{t}(x-v)+\varphi(t,x,v) \rangle.$$
\end{definition}
One important property of the weight $\mathbf{z}_{\mathrm{mod}}$ is that it is, by definition, constant along the nonlinear flow. In order to exploit this property, we need to prove that it does not deviate too much from the weight $e^t(x-v)$, which is preserved by the linear flow. 
\begin{lemma}\label{lemma_varphi}
We have $\T_\phi (\mathbf{z}_{\mathrm{mod}} )=0$. Moreover, the correction $\varphi$ satisfies the estimates
$$ \forall \, (t,x,v) \in [0,T) \times \R^2_x \times \R^2_v, \qquad |\varphi|(t,x,v) \lesssim \epsilon (1+t), \qquad |\nabla_v\varphi|(t,x,v) \lesssim \epsilon \, e^t .$$
\end{lemma}
\begin{proof}
The first property is straightforward, since we have defined $\varphi$ so that $\T_\phi (\mathbf{z}_{\mathrm{mod}} )=0$. Next, we have
$$  \left|\mathbf{T}_\phi (e^{t}(x^i-v^i))\right| \leq e^t|\nabla_x \phi  |(t,x) \lesssim \epsilon,$$ for $i\in\{1,2\}$. We then get that $|\mathbf{T}_\phi (\varphi)| \lesssim \epsilon$ on $[0,T) \times \R^2_x \times \R^2_v$, which implies, according to Lemma \ref{lem_characteristics_integrat_technic_pointwise}, the estimate for $\varphi$. In order to conclude the proof, it suffices to prove, as $2\partial_{v^i}=e^{-t} U_i-e^tS_i$, that there exists a constant $C>0$ such that
\begin{align}
 \forall \, (t,x,v) \in [0,T) \times \R^2_x \times \R^2_v, \qquad \qquad &|S_1\varphi|(t,x,v)+|S_2 \varphi|(t,x,v) \leq 2C \, \epsilon , \label{eq:boot1} \\
 & |U_1\varphi|(t,x,v)+|U_2 \varphi|(t,x,v)  \leq 2C \, \epsilon \, (1+t) . \label{eq:boot2}
 \end{align}
 By continuity, there exists a maximal time $0<T_{\mathrm{boot}} \leq T$ such that \eqref{eq:boot1}-\eqref{eq:boot2} holds on $[0,T_{\mathrm{boot}}) \times \R^2_x \times \R^2_v$. Let us prove by a bootstrap argument that $T_{\mathrm{boot}}=T$. Consider $Z \in \{ S_1, \, S_2, \, U_1, \, U_2 \}$ and apply the commutation formula of Lemma \ref{LemComfirstorder}. We get
\begin{align*}
 \T_\phi (Z \varphi)&=[\T_\phi, Z]+Z \T_\phi(\varphi)=\mu\nabla_x Z\phi \cdot \nabla_v \varphi-\mu e^t \nabla_x Z \phi \\
  & = \sum_{1 \leq i \leq 2}\mu \frac{e^{-t}}{2} \partial_{x^i} Z\phi \, U_i \varphi - \mu \frac{e^t}{2} \partial_{x^i} Z\phi \, S_i \varphi - \mu e^t \nabla_x Z \phi .
 \end{align*}
We then deduce from the pointwise decay estimates of Proposition \ref{proposition_estimate_phi} as well as from the bootstrap assumptions for the derivatives of $\varphi$ that, for all $(t,x,v) \in [0,T_{\mathrm{boot}}) \times \R^2_x \times \R^2_v$ and any $i \in \{1,2 \}$,
\begin{align*}
 \left| \T_\phi (S_i \varphi) \right|(t,x,v) & \lesssim \epsilon  e^{-2t}\left(  C  \epsilon (1+t) e^{-2t}+C \epsilon +1 \right) , \\
  \left| \T_\phi (U_i \varphi) \right|(t,x,v)& \lesssim \epsilon \big( C  \epsilon (1+t)e^{-2t}+C  \epsilon +  1 \big) .
  \end{align*}
Hence, if $C$ is chosen large enough and if $\epsilon$ is small enough, we have 
$$ \left| \T_\phi (S_i \varphi) \right|(t,x,v) \leq C\epsilon \, e^{-2t}  , \qquad \left| \T_\phi (U_i \varphi) \right|(t,x,v) \leq C\epsilon.$$
Using that $\varphi$ initially vanishes and Lemma \ref{lem_characteristics_integrat_technic_pointwise}, we improve \eqref{eq:boot1}-\eqref{eq:boot2} on $[0,T_{\mathrm{boot}})$, implying that $T_{\mathrm{boot}}=T$ as well the stated estimate for $\nabla_v \varphi$.
\end{proof}

\subsection{Pointwise estimates for the distribution and its derivatives}

We are now able to prove upper bounds for the weighted derivatives $\langle e^{-t}(x+v) \rangle^{M}  \mathbf{z}_{\mathrm{mod}}^{M} Z^{\beta}f$. We recall that for a multi-index $\beta$, the number of stable and unstable vector fields composing $Z^{\beta}$ are denoted by $\beta_s$ and $\beta_u$, respectively.

\begin{proposition}\label{prop_point_bound_deriv_distrib}
If $\e$ is small enough, then, for every $(t,x,v)\in [0,T)\times \R^2_x\times \R^2_v$, we have
\begin{alignat}{2}\label{estimate_point_upper_bound_distribution}
\langle e^{-t}(x+v) \rangle^{M} \big| \mathbf{z}_{\mathrm{mod}}^M Z^{\beta}f \big|(t,x,v)& \leq 2\e (1+t)^{\beta_u}, \qquad &&\text{if} \ |\beta| \leq N-1, \\
\langle e^{-t}(x+v) \rangle^{M} \big| \mathbf{z}_{\mathrm{mod}}^M Z^{\beta}f \big|(t,x,v)& \leq 2\e e^{\sigma t},  \qquad &&\text{if}\ 
|\beta| = N. \label{estimate_pointtoporder}
\end{alignat}
\end{proposition}

\begin{proof}
There exists a maximal time $T_0 \leq T$ such that \eqref{estimate_point_upper_bound_distribution} holds on $[0,T_0) \times \R^2_x\times\R^2_v$. By the initial data assumptions the estimate \eqref{estimate_point_upper_bound_distribution} holds when $t=0$, and then $T_0>0$. In the following, we prove that \eqref{estimate_point_upper_bound_distribution} holds with $\e+C\e^2$ instead of $\e$, where $C>0$ is a large constant. We improve the bootstrap assumption by using the method of characteristics through Lemma \ref{lem_characteristics_integrat_technic_pointwise}. For this purpose, we estimate, for every $|\beta| \leq N$,
\begin{align}
\T_{\phi}\left( \langle e^{-t}(x+v) \rangle^{M} \z_{\mathrm{mod}}^M Z^{\beta}f \right)&=M\T_\phi \left( \langle e^{-t}(x+v) \rangle \right) \langle e^{-t}(x+v) \rangle^{M-1} \z_{\mathrm{mod}}^M Z^{\beta}f \nonumber \\
& \quad +\langle e^{-t}(x+v) \rangle^{M} \z_{\mathrm{mod}}^M\T_{\phi}(Z^{\beta}f), \label{formula_transport_applied_weighted_distribution}
\end{align}
where we used $\T_\phi(\mathbf{z}_{\mathrm{mod}})=0$. We start by dealing with the first term on the right hand side of \eqref{formula_transport_applied_weighted_distribution}. We recall that $\T_{0}(e^{-t}(x+v))=0$, so that
$$
\left| \T_{\phi} \left( \langle  e^{-t}(x+v) \rangle  \right) \right|=| \nabla_x\phi (t,x) \cdot \nabla_v \langle e^{-t}(x+v) \rangle| \leq e^{-t} |\nabla_x \phi|(t,x).
$$
In view of the decay estimate for the force field given by Proposition \ref{proposition_estimate_phi} and the bootstrap assumption \eqref{estimate_point_upper_bound_distribution}, we get
\begin{align}
\nonumber &\left| \T_\phi \left( \langle e^{-t}(x+v) \rangle \right) \right| \langle e^{-t}(x+v) \rangle^{M-1}  |\z_{\mathrm{mod}}^M Z^{\beta}f |(t,x,v)  \\
\nonumber & \qquad \qquad \lesssim \e \, e^{-2t} \langle e^{-t}(x+v) \rangle^{M-1}  |\z_{\mathrm{mod}}^M Z^{\beta}f |(t,x,v)\\
&\qquad \qquad\lesssim \epsilon^2 (1+t)^{\beta_u} e^{-2t} \lesssim \epsilon^2 e^{-t}.  \label{eq:step1}
\end{align}
Next, we estimate the second term on the right hand side of \eqref{formula_transport_applied_weighted_distribution}. By the commutation formula in Lemma \ref{lemma_commuted_nonliner_Vlasov} and $2 \nabla_v = e^{-t}U+e^t S$, we have 
\begin{align*}
\big|\T_{\phi} \big(Z^{\beta}f \big) \big|&\lesssim \sum_{|\alpha|\leq |\beta|-1} \, \sum_{|\gamma|+|\alpha|\leq |\beta|} \, \sum_{ \alpha_u \leq \beta_u } |\nabla_x Z^{\gamma}\phi\cdot  \nabla_v Z^{\alpha}f|\\
&\lesssim  \sum_{|\alpha|\leq |\beta|-1} \, \sum_{|\gamma|+|\alpha|\leq |\beta|} \, \sum_{ \alpha_u \leq \beta_u } e^{-t}|\nabla_x Z^{\gamma}\phi\cdot  U Z^{\alpha}f|+ e^t|\nabla_x Z^{\gamma}\phi\cdot  S Z^{\alpha}f|,
\end{align*}
where the extra conditition $\alpha_u = \beta_u$ implies $\gamma_s \geq 1$ holds. The first term on the right hand side of the previous inequality are the easiest to handle, since they carry the factor $e^{-t}$. According to Proposition \ref{proposition_estimate_phi},  we have $|\nabla_x Z^\gamma\phi| \lesssim \epsilon \, e^{-t/2}$, so that
$$ \big|\T_{\phi} \big(Z^{\beta}f \big) \big| \lesssim \sum_{|\kappa| \leq N} \epsilon \, e^{-\frac{3t}{2}} |Z^\kappa f|+ \sum_{|\alpha|\leq |\beta|-1} \, \sum_{|\gamma|+|\alpha|\leq |\beta|} \, \sum_{ \alpha_u \leq \beta_u}  e^t|\nabla_x Z^{\gamma}\phi\cdot  S Z^{\alpha}f|.$$
Fix multi-indices $\alpha$ and $\gamma$ satisfying $|\alpha| \leq |\beta|-1$, $|\alpha|+|\gamma| \leq |\beta|$, $\alpha_u \leq \beta_u$ and $\gamma_s \geq 1$ if $\alpha_u =\beta_u$. Note that $Z^\xi=S_i Z^\alpha$ and $Z^\alpha$ contain the same number of unstable vector fields, we have $|\xi|=|\alpha|+1$ and $\xi_u=\alpha_u$.
\begin{itemize}
\item If $|\gamma| \leq N-1$, then $|\nabla_x Z^\gamma \phi|(t,x) \lesssim \epsilon \,  e^{-t-2\gamma_st}$ according to the pointwise decay estimates given by Proposition \ref{proposition_estimate_phi}. By exploiting the extra condition on $\alpha_u$, $\beta_u$ and $\gamma_s$, we get
$$ e^t|\nabla_x Z^{\gamma}\phi\cdot  S Z^{\alpha}f| \lesssim \epsilon \, \sup_{|\xi|=|\alpha|+1, \; \xi_u < \beta_u} | Z^\xi f|+\epsilon e^{-2t}\, \sup_{|\xi|=|\alpha|+1, \; \xi_u = \beta_u} | Z^\xi f| .$$
\item Otherwise $|\gamma| =N$, so $|\alpha|=0$ and we merely have $|\nabla_x Z^\gamma \phi|(t,x) \lesssim \epsilon \, e^{-(1-\sigma)t}$. Hence,
$$  e^t|\nabla_x Z^{\gamma}\phi\cdot  SZ^{\alpha}f| \lesssim \epsilon \,e^{\sigma t }|Sf|.$$
\end{itemize}
Let us focus now on the case $|\beta| \leq N-1$. Since $|\gamma| =N$ cannot occur, we have
$$ \big|\T_{\phi} \big(Z^{\beta}f \big) \big| \lesssim \sum_{|\kappa| \leq N} \epsilon \, e^{-\frac{3}{2}t} |Z^\kappa f|+ \sum_{|\xi| \leq N} \, \sum_{\xi_u < \beta_u}  \epsilon  | Z^\xi f|,$$
where, by convention, the second term vanish if $\beta_u=0$. Using now the bootstrap assumption \eqref{estimate_point_upper_bound_distribution}, we get
$$ \langle e^{-t}(x+v) \rangle^{M} \big| \z_{\mathrm{mod}}^M\T_{\phi}(Z^{\beta}f) \big|(t,x,v) \lesssim \left\{
    \begin{array}{ll}
    \epsilon^2 e^{-t} & \mbox{if $\beta_u=0$, } \\
     \epsilon^2 (1+t)^{\beta_u-1}  & \mbox{if $\beta_u \geq 1$}.
    \end{array}
\right.
$$
Combining \eqref{formula_transport_applied_weighted_distribution}-\eqref{eq:step1} and the last estimate, it gives
$$ \left| \T_{\phi}\left( \langle e^{-t}(x+v) \rangle^{M} \z_{\mathrm{mod}}^M Z^{\beta}f \right) \right| \lesssim  \left\{
    \begin{array}{ll}
    \epsilon^2 e^{-t} & \mbox{if $\beta_u=0$, } \\
     \epsilon^2 (1+t)^{\beta_u-1}  & \mbox{if $\beta_u \geq 1$}.
    \end{array}
\right.
$$
Then, we deduce by Lemma \ref{lem_characteristics_integrat_technic_pointwise} that there exists a constant $C > 0$ independent of $\e$ such that, for any $|\beta| \leq N-1$, $$\langle e^{-t}(x+v) \rangle^{M}|\mathbf{z}_{\mathrm{mod}}^{M}Z^{\beta}f|(t,x,v)\lesssim \|\langle x+v \rangle^{M}\mathbf{z}_{\mathrm{mod}}^{M}Z^{\beta}f(0)\|_{L^{\infty}_{x,v}}+C\e^2(1+t)^{\beta_u}$$ for every $(t,x,v)\in [0,T_0)\times\R^2_x\times\R^2_v$. We improve the bootstrap assumption by setting $\e$ small enough so that $C\e<1$. It remains us to deal with the case $|\beta|=N$. We have
$$ \big|\T_{\phi} \big(Z^{\beta}f \big) \big| \lesssim \epsilon e^{\sigma t }|Sf|+ \sum_{|\kappa| \leq N} \epsilon \, e^{-\frac{3}{2}t} |Z^\kappa f|+ \sum_{|\xi| \leq N} \, \sum_{\xi_u < \beta_u}  \epsilon  | Z^\xi f|.$$
We then get from \eqref{formula_transport_applied_weighted_distribution}, \eqref{eq:step1} and the bootstrap assumptions \eqref{estimate_point_upper_bound_distribution}--\eqref{estimate_pointtoporder},
$$ \left| \T_{\phi}\left( \langle e^{-t}(x+v) \rangle^{M} \z_{\mathrm{mod}}^M Z^{\beta}f \right) \right| \lesssim \epsilon^2 e^{\sigma t}+\epsilon^2 e^{-\frac{3}{2}t+\sigma t}+ \epsilon^2 e^{\sigma t} \lesssim \epsilon^2 e^{\sigma t}.$$
Applying once again Lemma \ref{lem_characteristics_integrat_technic_pointwise}, one can improve \eqref{estimate_pointtoporder} if $\epsilon$ is small enough.
\end{proof}

\subsection{Uniform boundedness of normalized stable averages} \label{subsection_unif_bound_stable_average}

Let $\beta$ be a multi-index. We proceed to show uniform boundedness in time of the \emph{normalized stable averages}, defined by 
\begin{equation}\label{introduce_stable_averages_for_decay}
Q^{\beta}(t,u):=e^{2t}\int _{\R^2_s} Z^{\beta}\bar{f}(t,s,e^tu)\mathrm{d}s.
\end{equation}
For this purpose, we begin studying the transport equation satisfied by the normalized stable average $Q^{\beta}(t,e^{-t}u)=e^{2t}\int _{\R^2_s} Z^{\beta}\bar{f}(t,s,u)\mathrm{d}s$.

\begin{proposition}\label{prop_derivative_time_stable_average}
Let $|\beta|\leq N-1$. Then, for every $(t,u)\in [0,T)\times \R^2_u$, we have $$ \Big|(\partial_t+u\cdot \nabla_u)\Big[e^{2t}\int _{\R^2_s}  Z^{\beta}\bar{f}(t,s,u) \mathrm{d}s\Big]\Big|\lesssim \e\dfrac{ \, (1+t)^3}{e^{2t}}\sup_{|\kappa|\leq |\beta|+1}\sup_{s\in\R^2_s }\big|\mathbf{z}^{3}_{\mathrm{mod}}Z^{\kappa}\bar{f}(t,s,u)\big|.$$
\end{proposition}

\begin{proof}
Fix $(t,u)\in  [0,T)\times \R^2_u$ and $|\beta|\leq N-1$. Integrating the commutation formula of Proposition \ref{lemma_commuted_nonliner_Vlasov} for $Z^{\beta}f$ and performing integration by parts in $s$, we have,
\begin{align*}
\frac{d}{dt}\Big[e^{2t}\int _{\R^2_s}  Z^{\beta}\bar{f} \mathrm{d}s\Big]&=2e^{2t}\int _{\R^2_s} Z^{\beta}\bar{f} \mathrm{d}s+e^{2t}\int _{\R^2_s} \partial_t  Z^{\beta}\bar{f} \mathrm{d} s\\
&=-e^{2t}\int _{\R^2_s} u\cdot \nabla_u  Z^{\beta}\bar{f} \mathrm{d}s+\mu \int _{\R^2_s} e^{2t}\nabla_x\phi\cdot \nabla_v Z^{\beta}\bar{f}  \mathrm{d}s\\
&\qquad +\mu \sum_{|\delta|+|\gamma|\leq |\beta|}\sum_{|\gamma|\leq |\beta|-1} C^{\beta}_{\gamma \delta}\int _{\R^2_s}  e^{2t}\nabla_x Z^{\delta}\phi\cdot  \nabla_v Z^{\gamma}\bar{f} \mathrm{d}s.
\end{align*} 
Decomposing the vector fields $2\partial_{v^i}=\partial_{u^i}-\partial_{s^i}$ and using integration by parts, we obtain
\begin{align*}
\Big|(\partial_t+u\cdot \nabla_u)\Big[e^{2t}\int_{\R^2_{s}} Z^{\beta}\bar{f} (t,s,u) \mathrm{d}s\Big]\Big|\lesssim \sum_{|\delta|+|\gamma|\leq |\beta|}\Big| \int _{\R^2_s}  e^{2t}\nabla_x Z^{\delta}\phi\cdot  \big(\nabla_u Z^{\gamma}\bar{f}-\nabla_s Z^{\gamma}\bar{f} \big)\mathrm{d}s\Big| &\\
\lesssim \sum_{|\delta|+|\gamma|\leq |\beta|} \int _{\R^2_s}  \big|e^{t}\nabla_x Z^{\delta}\phi\cdot  e^{t}\nabla_u Z^{\gamma}\bar{f}|+ |(e^{t}\nabla_x)^2 Z^{\delta}\phi\cdot  Z^{\gamma}\bar{f} \big|\mathrm{d}&s.
\end{align*}
Next, we use the time decay of the force field and the linear weight $\langle e^t s \rangle^3$ to obtain $$ \Big|(\partial_t+u\cdot \nabla_u)\Big[e^{2t}\int _{\R^2_s} Z^{\beta}\bar{f}(t,s,u)\mathrm{d}s\Big]\Big|\lesssim \e\sup_{|\kappa|\leq |\beta|+1}\sup_{s\in  \R_s^2}|\langle e^t s \rangle^{3}Z^{\beta}\bar{f}(t,s,u)|\int_{\R^2_s}\frac{\mathrm{d}s}{\langle e^t s \rangle^3}.$$ Finally, we obtain the result by making use of $$\int_{\R^2_s}\frac{\mathrm{d}s}{\langle e^t s \rangle^3}\leq \int_{\R^2_s}\frac{\mathrm{d}s}{(1+|e^{t}s|^2)^{\frac{3}{2}}}=\frac{1}{e^{2t}}\int_{\R^2_y}\frac{\mathrm{d}y}{(1+|y|^2)^{\frac{3}{2}}}\leq \frac{2}{e^{2t}}$$
as well as $\langle e^t s \rangle \leq 2(1+t) \mathbf{z}_{\mathrm{mod}}$, which is implied by Lemma \ref{lemma_varphi} and $2s=x-v$.
\end{proof}

In particular, we obtain uniform boundedness of the normalized stable averages \eqref{introduce_stable_averages_for_decay} of the distribution function in a weighted $L^\infty_u$ space. We recall that $e^{-t}(x+v)=2e^{-t}u$.

\begin{corollary}\label{cor_bound_small_integral_stable}
For every $(t,u)\in [0,T)\times \R_u^2$ and every $|\beta|\leq N-1$, we have $$ \Big| \langle e^{-t} u \rangle^{M} \, e^{2t}\int_{\R^2_s}Z^{\beta}\bar{f}(t,s,u)\mathrm{d}s \Big|\lesssim \e.$$
\end{corollary}

\begin{proof}
Applying a variant of Lemma \ref{lem_characteristics_integrat_technic_pointwise} for the transport operator $\partial_t+u\cdot \nabla_u$ combined with Proposition \ref{prop_derivative_time_stable_average}, we have that for every $|\beta|\leq N-1$ and every $(t,u)\in [0,T)\times \R^2_u$,
\begin{align*}
\Big| \langle e^{-t} u \rangle^{M} e^{2t }\int_{\R^2_s}Z^{\beta}\bar{f}(t,s,u)\mathrm{d}s\Big| & \lesssim \epsilon+\int_0^{t}\frac{\langle \tau \rangle^{3}}{e^{2\tau}} \sup_{|\kappa|\leq N}\sup_{s\in\R^2_s }\langle e^{-\tau} u \rangle^{M}\big|\mathbf{z}^{3}_{\mathrm{mod}}Z^{\kappa}\bar{f}(\tau,s,u)\big| \mathrm{d} \tau \\
& \lesssim \epsilon+ \e\int_0^{t}\frac{\langle \tau \rangle^3 e^{\sigma \tau}}{e^{2\tau}}d\tau\lesssim \e,
\end{align*}
where we have used Proposition \ref{prop_point_bound_deriv_distrib} to bound the integrand.
\end{proof}

Thus, we obtain control in time of the normalized stable average $Q^{\beta}(t,u)$, by integrating the transport equation satisfied by $Q^{\beta}(t,e^{-t}u)=e^{2t}\int _{\R^2_s} Z^{\beta}\bar{f}(t,s,u)\mathrm{d}s$ and applying Proposition \ref{prop_point_bound_deriv_distrib}.

\begin{proposition}\label{prop_derivative_time_stable_average_new_normalisation_first}
Let $|\beta|\leq N-1$. Then, for every $0\leq t_1\leq t_2\leq T$ and every $u\in \R^2_u$, we have 
\begin{alignat}{2}
\langle u \rangle^M \Big|e^{2t_2} \!\int _{\R^2_s} Z^{\beta}\bar{f}(t_2,s,e^{t_2}u)\mathrm{d}s-e^{2t_1} \! \int _{\R^2_s} Z^{\beta}\bar{f}(t_1,s,e^{t_1}u)\mathrm{d}s\Big|&\lesssim \e \! \int_{t_1}^{t_2} \! \dfrac{\langle \tau \rangle^{4+|\beta|}}{e^{2\tau}}\mathrm{d} \tau,\quad  && 
|\beta| \leq N-2, \nonumber\\
\langle u \rangle^M \Big|e^{2t_2}\!\int _{\R^2_s} Z^{\beta}\bar{f}(t_2,s,e^{t_2}u)\mathrm{d}s-e^{2t_1} \! \int _{\R^2_s} Z^{\beta}\bar{f}(t_1,s,e^{t_1}u)\mathrm{d}s\Big|&\lesssim \e \!\int_{t_1}^{t_2} \! \dfrac{\mathrm{d} \tau}{e^{(2-\sigma)\tau}},\quad && 
|\beta| = N-1.\nonumber
\end{alignat}
\end{proposition}

As a result, or by a direct application of Corollary \ref{cor_bound_small_integral_stable}, we obtain the desired uniform boundedness in time of the normalized stable averages $Q^{\beta}(t,u)$.

\begin{corollary}
Let $|\beta|\leq N-1$. Then, for every $(t,u)\in [0,T)\times \R^2_u$, we have $$ \Big|e^{2t}\int _{\R^2_s} Z^{\beta}\bar{f}(t,s,e^tu)\mathrm{d}s\Big|\lesssim \e.$$
\end{corollary}

\subsection{Pointwise decay estimates for velocity averages}

In this subsection, we prove that the decay rate of $\rho(Z^{\beta}f)$ for $|\beta| \leq N -1$ coincides with the one of the linearized system. In particular, we improve the bootstrap assumption \eqref{boot2}. The starting point consists in performing the change of variables $y=\frac{1}{2}e^t(x-v)$.

\begin{lemma}\label{lemma_change_variables_decay}
Let $g: [0,T)\times \R^2_x\times \R^2_v\to\R$ be a sufficiently regular distribution. Then, for every $(t,x)\in [0,T)\times\R^2_x$, we have $$e^{2t}\int_{\R_v^2}g(t,X_{\L}(-t),V_{\L}(-t))\mathrm{d}v=\int_{\R^2_y} \bar{g}\Big(t,y,\frac{1}{e^t}\Big(x-\frac{y}{e^t}\Big)\Big)\mathrm{d}y.$$
\end{lemma}

This change of variables is motivated by the linearized problem. Every solution to the Vlasov equation with the potential $\frac{-|x|^2}{2}$ is transported along the lines of  corresponding characteristic flow, so $h(t,X_{\L}(t),V_{\L}(t))=h(0,x,v)$. The previous lemma, applied to $g(t,x,v)=h(0,x,v)$, shows that $|\rho (h)|(t,x)\lesssim e^{-2t}$. Next, we control $\rho( |Z^{\beta}f|)$ for every $|\beta|\leq N$, which has a slower decay rate than in the linear case.
\begin{lemma}\label{Lemdecay1}
Let $g:[0,T) \times \R^2_x \times \R^2_v \to \R$ be a sufficiently regular function. Then, for all $(t,x) \in [0,T) \times \R^2_x$,
$$  \int_{\R^2_v} |g|(t,x,v) \mathrm{d} v \lesssim \frac{1}{(e^t+|x|)^2}\sup_{(y,v) \in \R^2 \times \R^2} \big|  \mathbf{z}_{\mathrm{mod}}^5 \, g \big|(t,y,v)+\langle e^{-t}(y+v) \rangle^5 |g|(t,y,v)  .$$
\end{lemma}
\begin{proof}
We start by writing
$$ \int_{\R^2_v} |g(t,x,v)| \mathrm{d} v \leq \sup_{(y,v) \in \R^2 \times \R^2} \big| \mathbf{z}_{\mathrm{mod}}^3 \, g \big|(t,y,v) \int_{ \R^2_v} \frac{\mathrm{d} v}{\langle e^t(x-v)+\varphi(t,x,v) \rangle^3} .$$
The change of variables $y= e^t(x-v)$ yields 
$$
e^{2t}\int_{\R^2_v} |g(t,x,v)| \mathrm{d} v \lesssim \sup_{(y,v) \in \R^2 \times \R^2} \big| \mathbf{z}_{\mathrm{mod}}^3 \,  g \big|(t,y,v) \int_{\R^2_y} \frac{\mathrm{d} y}{\langle y+\varphi (t,x, x-e^{-t}y ) \rangle^{3}}.
$$
Let $\phi : y \mapsto y+\varphi (t,x, x-e^{-t}y )$ and $\psi = \mathrm{Id}-\phi$. In order to perform the change of variables $w= \phi(y)$, we will prove that $\frac{1}{2 } \leq |\det( \mathrm{d} \phi(y))| \leq 2$ for all $y \in \R^2$. This property is satisfied because 
$$ \forall \, (t,x,y) \in [0,\infty) \times \R^2_x\times\R^2_y , \qquad |\mathrm{d} \psi (y) | \leq e^{-t} |\nabla_v \varphi | (t,x, x-e^{-t}y )\lesssim \epsilon \leq \frac{1}{2} ,$$
which holds provided that $\epsilon$ is small enough by Lemma \ref{lemma_varphi}. We then deduce $$ \int_{\R^2_y} \frac{\mathrm{d} y}{\langle y+\varphi (t,x, x-e^{-t}y ) \rangle^{3}} \leq  2\int_{\R^2_w} \frac{\mathrm{d} w}{\langle w \rangle^{3}} \leq 4.$$ 
We have then proved the estimate
\begin{equation}\label{eq:gtimedecay}
\forall \, (t,x) \in [0,T) \times \R^2_x , \qquad  \int_{\R^2_v} |g|(t,x,v) \mathrm{d} v \lesssim e^{-2t}\sup_{(y,v) \in \R^2 \times \R^2} \big|  \mathbf{z}_{\mathrm{mod}}^3 \, g \big|(t,y,v)  .
\end{equation}
In order to obtain the spatial decay and conclude the proof, remark that, in view of Lemma \ref{lemma_varphi},
\begin{align*}
 |x| &=  \left| e^t \big[e^{-t}(x+v) \big] + e^{-t}\big[e^t(x-v) \big] \right| \lesssim e^{t} \langle e^{-t}(x+v) \rangle+e^{-t} \mathbf{z}_{\mathrm{mod}}(t,x,v)+\epsilon \, (1+t)e^{-t}  \\
& \lesssim e^{t} \langle e^{-t}(x+v) \rangle+ \mathbf{z}_{\mathrm{mod}}(t,x,v) 
\end{align*}
and apply \eqref{eq:gtimedecay} to $\langle e^{-t}(x+v) \rangle^2 g$ as well as $\mathbf{z}_{\mathrm{mod}}^2g$.
\end{proof}
We are now able to improve the bootstrap assumption \eqref{boot2} if $C_{\mathrm{boot}}$ is chosen large enough. Applying Lemma \ref{Lemdecay1} to $g(t,x,v)= \langle e^{-t}(x-v) \rangle^{M-5} \mathbf{z}^{M-5}_{\mathrm{mod}} Z^{\beta}f(t,x,v)$ and then Proposition \ref{prop_point_bound_deriv_distrib}, we get the following estimates.
\begin{corollary}\label{cor_decay_losing_spatial_density}
For any $|\beta|\leq N$ and for every $(t,x)\in [0,T)\times \R^2_x$, we have $$ \int _{\R^2_v} \langle e^{-t}(x-v) \rangle^{M-5} \big|\mathbf{z}^{M-5}_{\mathrm{mod}} Z^{\beta}f \big|(t,x,v)\mathrm{d}v\lesssim \dfrac{\e \, e^{\sigma t}}{(e^{t}+|x|)^2} .$$
\end{corollary}

Our next goal will be to remove the $e^{\sigma t}$ loss in the estimate for $\rho (Z^{\beta}f)$ in Corollary \ref{cor_decay_losing_spatial_density}. We will not be able to improve the estimate of $\rho (Z^{\beta}f)$ for top order derivatives since our analysis relies on the following lemma, which requires a loss of one derivative. We remark that  spatial decay will only be exploited in the next Section \ref{section_modified_scattering}.

\begin{lemma}\label{lem_asympt_exp_spatial_density_estimate}
Let $g: [0,T)\times \R^2_x\times \R^2_v\to\R$ be a sufficiently regular distribution. Then, for every $t\in [0,T)$ we have 
\begin{align*}
\Big| e^{2t}\int_{\R_v^2}g(t,X_{\L}(-t),V_{\L}(-t))\mathrm{d}v~-&\int_{\R^2_y}\bar{g}\Big(t,y,\frac{x}{e^t}\Big)\mathrm{d}y\Big| \\
& \lesssim \frac{1}{(e^{t}+|x|)^2}\sup_{(s,u)\in\R^2_s\times\R^2_u} \big( \langle s \rangle+ \langle u \rangle \big)^6 |\nabla_u\bar{g}(t,s,u)|.
\end{align*}
\end{lemma}

\begin{proof}
By the mean value theorem, we have 
\begin{align*}
\Big| \bar{g}\Big(t,y,\frac{1}{e^t}\Big(x-\frac{y}{e^t}\Big)\Big)~-~&\bar{g}\Big(t,y,\frac{x}{e^t}\Big)\Big| \lesssim \frac{|y|}{e^{2t}}\sup_{\tau \in [0,1]} |\nabla_u \bar{g}|\Big(t,y,\frac{1}{e^t}\Big(x-\tau\frac{y}{e^t}\Big)\Big).
\end{align*}
Consider now $u_\tau:= e^{-t}(x-\tau e^{-t}y)$ and remark that
\begin{equation}\label{eq:defutau}
 \forall \, \tau \in [0,1], \qquad 2\langle u_\tau \rangle+2\langle y \rangle \geq 2 \langle e^{-t} x \rangle.
\end{equation}
Indeed, if $\tau e^{-t} |y| \leq \frac{|x|}{2}$ we have $2\langle u_\tau \rangle \geq \langle e^{-t} x \rangle$. Otherwise, there holds $2 \langle y \rangle \geq \langle e^t x \rangle$. We then deduce that
\begin{align*}
\Big| \int_{\R^2_y} \bar{g}\Big(t,y,\frac{1}{e^t}\Big(x-\frac{y}{e^t}&\Big)\Big)\mathrm{d}y-\int_{\R^2_y}\bar{g}\Big(t,y,\frac{x}{e^t}\Big)\mathrm{d}y\Big| \\
& \quad \lesssim \frac{1}{e^{2t}\langle e^{-t} x \rangle^2}\sup_{(s,u)\in\R^2_s\times\R^2_u} \big( \langle s \rangle+ \langle u \rangle \big)^2  \langle s \rangle^{4}|\nabla_u\bar{g}|(t,s,u)\int_{\R^2_y}\frac{\mathrm{d}y}{\langle y \rangle^{3}}.
\end{align*}
It remains to apply Lemma \ref{lemma_change_variables_decay} and to use $e^{t} \langle e^{-t} x \rangle \gtrsim e^t+|x|$. 
\end{proof}

Finally, we are able to prove optimal time decay of the velocity averages $\rho(Z^{\beta}f)$ for every $|\beta|\leq N-1$. The following proposition improves the bootstrap assumption \eqref{boot2} if the constant $C_{\mathrm{boot}}$ is chosen large enough. 

\begin{proposition}\label{Protimedecay}
For every $|\beta|\leq N-1$, the decay of the spatial density $\rho(Z^{\beta}f)$ is optimal. There exists $C>0$ such that for every $(t,x)\in [0,T) \times \R_x^2$, we have $$ \Big|\int_{\R^2_v} Z^{\beta}f(t,x,v)\mathrm{d}v\Big|\leq \frac{C\e}{(e^{t}+|x|)^2}.$$
\end{proposition}

\begin{proof}
Consider $g(t,x,v)= Z^{\beta}f(t,X_{\L}(t),V_{\L}(t))$. Applying Lemma \ref{lem_properties_linear_change_variable_trivial} and Proposition \ref{prop_point_bound_deriv_distrib} to $g$, we have 
\begin{align}
\sup_{(s,u)\in\R^2_s\times\R^2_u} \big(\langle s \rangle+\langle u \rangle)^{6}|\partial_u\bar{g}| (t,s,u)&\lesssim \sup_{(s,u)\in\R^2_s\times\R^2_u} \sum_{|\kappa|\leq 1}  \big(\langle e^t s \rangle+\langle e^{-t}u \rangle)^{6} \big|Z^{\kappa}Z^{\beta}\bar{f} \big|(t,s,u)  \nonumber\\ 
&\lesssim \left\{ \begin{array}{ll}
    \epsilon \, \langle t \rangle^{N+5} & \mbox{if $|\beta| \leq N-2$,} \\
     \e (1+t)^{6}e^{\sigma t}  & \mbox{if $|\beta|=N-1$},
    \end{array}
\right. \label{estimate_improving_BA_spatial_density}
\end{align}
since we have $M \geq 6$ and $\langle e^t s \rangle \lesssim (1+t)\bar{ \mathbf{z}}_{\mathrm{mod}}(t,s,u)$. Finally, applying Lemma \ref{lem_asympt_exp_spatial_density_estimate} to the distribution $g$, we have $$e^{2t}\Big|\int_{\R^2_v}  Z^{\beta}f(t,x,v)\mathrm{d}v\Big|\lesssim \Big|\int_{\R_y^2}\bar{g}\Big(t,y,\frac{x}{e^t}\Big)\mathrm{d}y\Big|+\e\dfrac{(1+t)^{6}e^{\sigma t}}{(e^{t}+|x|)^2},$$ for every $t\geq 0$. It remains to bound by $\e \langle e^{-t} x \rangle^{-2}$ the first term on the right hand side. For this purpose, we apply Corollary \ref{cor_bound_small_integral_stable} to obtain 
\begin{align*} 
\Big|\int_{\R_y^2}\bar{g}\Big(t,y,\frac{x}{e^t}\Big)\mathrm{d}y\Big| &=\Big|\int_{\R_y^2}  Z^{\beta}\bar{f}\Big(t,\frac{y}{e^t},x\Big)\mathrm{d}y\Big| =\Big|\langle e^{-t}x \rangle^{-2} e^{2t} \int_{\R_w^2} \langle e^{-t} x \rangle^2 Z^{\beta}\bar{f}(t,w,x)\mathrm{d}w\Big| \\
& \lesssim \e \langle e^{-t}x \rangle^{-2}.
\end{align*}
\end{proof}

\subsection{Improved estimates for derivatives of velocity averages}

As in the linear case, spatial derivatives of velocity averages $\partial_x \rho(f)$ enjoy stronger decay properties. 
\begin{proposition}
For every $|\beta|\leq N-1$, every $i\in \{1,2\}$, and every $(t,x)\in[0,T)\times \R^2_x$, we have
$$\Big|\int_{\R^2_v}\partial_{x^i}Z^{\beta}f(t,x,v)\mathrm{d}v\Big|\lesssim  \e\frac{e^{\sigma t}}{(e^t+|x|)^3}.$$
\end{proposition}

\begin{proof}
Writing the velocity averages in terms of the commuting vector fields, we have
\begin{align*}
\int_{\R^2_v}\partial_{x^i}Z^{\beta}f(t,x,v)\mathrm{d}v&=\dfrac{1}{e^t+|x|}\int_{\R^2_v}e^t\partial_{x^i}Z^{\beta}f+|x|\partial_{x^i}Z^{\beta}f \mathrm{d}v\\
&=\dfrac{1}{e^t+|x|}\int_{\R^2_v}U_iZ^{\beta}f \mathrm{d}v+\dfrac{1}{e^t+|x|}\sum_{j=1}^n\int_{\R^2_v}\frac{x^j}{|x|}R_{ij}Z^{\beta}f + \frac{x^i}{|x|} LZ^{\beta}f \mathrm{d}v.
\end{align*}
Finally, we use Corollary \ref{cor_decay_losing_spatial_density} to obtain
$$\Big|\int_{\R^2_v}\partial_{x^i}Z^{\beta}f(t,x,v)\mathrm{d}v\Big|\lesssim \dfrac{1}{e^t+|x|}\sum_{|\kappa|\leq 1}\int_{\R^2_v}|Z^{\kappa}Z^{\beta}f|\mathrm{d}v\lesssim \e \frac{e^{\sigma t}}{(e^t+|x|)^3}. $$
\end{proof}

\section{Modified scattering for the distribution function}\label{section_modified_scattering}

In this section, we obtain the modified scattering properties of the distribution function and its derivatives. For this purpose, we determine the self-similar profile of the spatial density $\rho(f)$, the self-similar profile of the force field $\nabla_x\phi$, and we define the modified trajectories along which $f$ converges to a new distribution function. In Section \ref{section_small_data_global_existence}, we obtained that the distribution function $f$ is global in time. As a result, all the statements proved in the previous sections hold true for every $t\in [0,\infty)$. During this section, we use several times the estimates obtained in Section \ref{section_small_data_global_existence}.

\subsection{Convergence of normalized stable averages}

In this subsection, we study the convergence in time of the \emph{normalized stable averages}, given by $$Q^{\beta}(t,u)=e^{2t}\int Z^{\beta}\bar{f}(t,s,e^tu) \mathrm{d}s.$$ We show that the normalized stable averages $Q^{\beta}(t,u)$ converge to regular real-valued functions $Q^{\beta}_{\infty}(u)$ defined on $\R^2_u$. Moreover, we will prove that the profile $Q^{\beta}_{\infty}(u)$ dictates the late-time asymptotic behavior of the normalized spatial density $e^{2t}\rho(Z^{\beta}f)$ (see Proposition \ref{prop_self_similar_spatial_density} below for more details). 

The analysis performed in the previous section shows that the distribution function $f$ is global in time. As a result, the control in time of the normalized stable averages in Proposition \ref{prop_derivative_time_stable_average_new_normalisation_first} holds for every $t\in [0,\infty)$. Thus, we obtain the existence of the profiles $Q_{\infty}^{\beta}(u)$. 

\begin{proposition} \label{prop_conver_unstable_averages}
Let $|\beta|\leq N-2$. There exists a continuous function $Q_{\infty}^{\beta}\in L^{\infty}_u \cap L^1_u$ such that for every $(t,u)\in [0,\infty)\times \R^2_u$, we have
\begin{equation}\label{estimate_prop_conv_unstable_average}
\Big|\langle u \rangle^{M} \Big(Q_{\infty}^{\beta}(u)-e^{2t}\int_{\R^2_{s}}Z^{\beta}f(t,s,e^tu)\mathrm{d}s \Big) \Big| \leq \e\frac{(1+t)^{N+2}}{e^{2t}}.
\end{equation}
\end{proposition}

\begin{proof}
Since Proposition \ref{prop_derivative_time_stable_average_new_normalisation_first} is in fact verified for all $0 \leq t_1 \leq t_2 <+\infty$, there exists $Q_{\infty}^{\beta}\in L^{\infty}_u$ such that \eqref{estimate_prop_conv_unstable_average} holds. As $M >2$, we have $Q_{\infty}^{\beta}\in L^1_u$ as well. 
\end{proof}

\begin{remark}
Performing the change of variables $(x,v)\mapsto (s,u)$, the total mass of the system can be written as $$\|f\|_{L^1_{x,v}}=2\|\bar{f}\|_{L^1_{s,u}}=2\int_{\R^2_u}\int_{\R^2_s}\bar{f}(t,s,u)\mathrm{d}s\mathrm{d}u=2\int_{\R^2_u}e^{2t}\int_{\R^2_s}\bar{f}(t,s,e^tu)\mathrm{d}s\mathrm{d}u.$$ We observe that the function $Q_{\infty}:\R^2_u\to[0,\infty)$ determines the normalized total mass $Q_{\infty}(u_0)$ of the system in the stable leaves $\{u=u_0\}$. 
\end{remark}

We are now able to show that the asymptotic behavior of the spatial density $\rho(Z^{\beta}f)$ is described by a self similar asymptotic profile in terms of the normalized stable average $Q_{\infty}^{\beta}$.

\begin{proposition}\label{prop_self_similar_spatial_density}
For every $|\beta| \leq N-2$, the spatial density $\rho(Z^{\beta}f)$ satisfies $$\Big|e^{2t}\rho(Z^{\beta}f)(t,x)-Q_{\infty}^{\beta}\Big(\frac{x}{e^t}\Big)\Big|\lesssim \e\frac{(1+t)^{N+5}}{(e^t+|x|)^2}.$$
\end{proposition}

\begin{proof}
Let $|\beta|\leq N-2$. Applying Lemma \ref{lem_asympt_exp_spatial_density_estimate} and the estimate \eqref{estimate_improving_BA_spatial_density} to $g(t,x,v):=Z^{\beta}f(t,X_{\L}(t),V_{\L}(t))$, we obtain 
\begin{equation}\label{estimate_asymp_beh_spatial_density_self_similar}
\Big| e^{2t}\int_{\R^2_v} Z^{\beta}f(t,x,v)\mathrm{d}v-e^{2t}\int_{\R^2_y} Z^{\beta}\bar{f}\Big(t,y,\Big(\frac{x}{e^t}\Big)e^t\Big)\mathrm{d}y \Big|\lesssim \e\dfrac{\langle t \rangle^{N+5}}{(e^{t}+|x|)^2}.
\end{equation}
Furthermore, we obtain from Proposition \ref{prop_conver_unstable_averages} that $$\Big|Q^{\beta}_{\infty}(u)-e^{2t}\int_{\R^2_s} Z^{\beta}\bar{f}(t,s,e^tu)\mathrm{d}s\Big|\lesssim \e\frac{\langle t \rangle^{N+2}}{e^{2t} \langle u \rangle^{2}},$$ as $M-3\geq 2$. The result follows from \eqref{estimate_asymp_beh_spatial_density_self_similar} and the last estimate applied for $u=e^{-t}x$.
\end{proof}

\subsection{Self-similar asymptotic profile of the force field}\label{subsection_self_similar_asymp_profile}

We fix, for all this section, a sufficiently regular function $Q_\infty : \R^2_u \to [0,\infty)$. Motivated by Proposition \ref{prop_self_similar_spatial_density}, we define the \emph{asymptotic potential} $\phi_{\mathrm{asymp}}$ as the unique solution to the \emph{asymptotic Poisson equation} given by
\begin{equation}\label{eq:defasympphi}
\Delta_u \, \phi_{\mathrm{asymp}}\big[Q_\infty \big] = Q_\infty.
 \end{equation}
We show the dependence of the asymptotic potential $\phi_{\mathrm{asymp}}$ on the RHS of the asymptotic Poisson equation by $\phi_{\mathrm{asymp}}\big[Q_\infty \big]$.

Let us recall the set of vector fields $\lambda_u:=\{ \partial_{u^i}, L_u, R^u_{ij,u}  \}$ introduced in Subsection \ref{subsection_macro_micro_unst_vectors}. We consider an ordering on $\lambda_{u}$, which is compatible with the one on $\lambda$ and we denote by $Z_u^\beta$ the vector field $Z_u^{\beta_1} \dots Z^{\beta_p}_u$, where $|\beta|=p$. 

Before commuting the asymptotic Poisson equation, we prove that $Q_\infty$ is differentiable and we relate its derivatives to $Q^\beta_\infty$. 

\begin{proposition}\label{prop_properties_derivatives_self_similar_profile_spatial_density}
For any $|\beta|\leq N-1$ the function $Q^{\beta}_{\infty}$ is of class $C^{N-1-|\beta|}(\R^2_u)$. Moreover, the derivatives of $Q^{\beta}_{\infty}$ can be obtained by iterating the following relations.
\begin{enumerate}[label=(\alph*)]
\item if $\beta_s\geq 1$, so that $Z^{\beta}$ is composed by at least one stable vector field, we have $Q_{\infty}^{\beta}=0$;
\item if $Z^\beta= U_i Z^{\kappa}$, we have $Q_\infty^\beta = \partial_{u^i} Q^\kappa_\infty$;
\item if $Z^{\beta}=LZ^{\kappa}$, we have $Q_{\infty}^{\beta}=L_uQ^{\kappa}_{\infty}-2Q^{\kappa}_{\infty};$
\item if $Z^{\beta}=R_{12}Z^{\kappa}$, we have $Q^\beta_\infty=R_{12,u}Q^{\kappa}_{\infty}$.
\end{enumerate}
\end{proposition}

\begin{proof}
First, we assume that $\beta_s\geq 1$. Since $[Z,S_i]$ either vanishes or it is another stable vector field, then, it suffices to consider the case when $Z^{\beta}=S_iZ^{\xi}$. In this case, we can integrate by parts to show that $\int  Z^{\beta}\bar{f}\mathrm{d}s=0$, so $Q^{\beta}_{\infty}=0$. Next, for every $(t,u)\in [0,\infty)\times \R^2_u$, we have $$\partial_{u^i} e^{2t}\int_{\R^2_{s}}Z^{\kappa}\bar{f}(t,s,e^tu)\mathrm{d}s= e^{2t}\int_{\R^2_{s}}e^t\partial_{u^i}Z^{\kappa}\bar{f}(t,s,e^tu)\mathrm{d}s= e^{2t}\int_{\R^2_{s}}U_iZ^{\kappa}\bar{f}(t,s,e^tu)\mathrm{d}s.$$ According to Proposition \ref{prop_conver_unstable_averages}, the right hand side converges to $Q_{\infty}^{\beta}$ in $L^{\infty}_u$ as $t\to\infty$. If $Z^\beta =LZ^\kappa$, it is enough to notice $$e^{2t}\int  LZ^{\kappa}\bar{f}\mathrm{d}s=e^{2t} \big(u^1\partial_{u^1}+u^2 \partial_{u^2} \big)\int_{\R^2_s}  Z^{\kappa}\bar{f}\mathrm{d}s-2e^{2t}\int_{\R^2_s}  Z^{\kappa}\bar{f}\mathrm{d}s,$$ where we used that $L=u^1\partial_{u^1}+u^2\partial_{u^2}+s^1\partial_{s^1}+s^2\partial_{s^2}$. It remains to apply the Proposition \ref{prop_conver_unstable_averages}. Finally, it is enough to notice $$e^{2t}\int_{\R^2_s}  R_{12}Z^{\kappa}\bar{f}\mathrm{d}s=e^{2t} \big(u^1\partial_{u^2}-u^2 \partial_{u^1} \big)\int_{\R^2_s} Z^{\kappa}\bar{f}\mathrm{d}s,$$ where we used that $R_{12}=u^1\partial_{u^2}-u^2\partial_{u^1}+s^1\partial_{s^2}-s^2\partial_{s^1}$, and to apply Proposition \ref{prop_conver_unstable_averages}.
\end{proof}
We have the following commutation relations.
\begin{proposition}\label{Proasymp}
There holds, for any $i \in \{1,2 \}$,
\begin{align*}
 \Delta_u \, \partial_{u^i} \phi_{\mathrm{asymp}}\big[Q_\infty \big] = \partial_{u^i} Q_\infty,& \qquad \Delta_u \, R_{12,u} \phi_{\mathrm{asymp}}\big[Q_\infty \big]= R_{12,u} Q_\infty, \\
 \Delta_u \, L_u \phi_{\mathrm{asymp}}\big[Q_\infty \big]&= L_u Q_\infty+2Q_\infty .
\end{align*}
In particular, for any $Z_u \in \lambda_{u}$ we have
$$ Z_u \phi_{\mathrm{asymp}} [Q_\infty ] = \phi_{\mathrm{asymp}} [Z_u Q_\infty]+2\delta^{Z_u}_{L_u}\phi_{\mathrm{asymp}} [Q_\infty ].$$ More generally, for any multi-index $\beta$, there exist $C^\beta_{\alpha} \in \mathbb{N}$ such that
\begin{equation}\label{identity_commut_unstable_asymp_poisson}
 \Delta_u \, Z_u^\beta \phi_{\mathrm{asymp}}\big[Q_\infty \big] = \sum_{|\alpha| \leq |\beta|} C^\beta_\alpha Z_u^\alpha Q_\infty, \qquad Z_u^\beta \phi_{\mathrm{asymp}}\big[Q_\infty \big] = \sum_{|\alpha| \leq |\beta|} C^\beta_\alpha \phi_{\mathrm{asymp}} \big[ Z_u^\alpha Q_\infty \big].
\end{equation}
\end{proposition}
\begin{proof}
One simply has to iterate the commutation relations
$$ [\Delta_u, \partial_{u^i}]=[\Delta_u, R^u_{12}]=0, \qquad [\Delta_u , L_u]=2\Delta_u.$$
\end{proof}

\begin{remark}
We observe that the coefficients $C_{\alpha}^{\beta}\in\N$ in the commuted asymptotic Poisson equation \eqref{identity_commut_unstable_asymp_poisson}, are equal to the coefficients $C_{\alpha}^{\beta}\in\N$ in the commuted Poisson equation \eqref{identity_commuted_poisson_eqn}. In the following, this property is key in order to obtain the self-similar asymptotic profile of the commuted force field $\nabla_x Z^\gamma \phi$.
\end{remark}

We are now able to show that the asymptotic behavior of the force field $\nabla_x\phi$ is described by a self similar asymptotic profile in terms of the asymptotic potential $\phi_{\mathrm{asymp}}$.

\begin{proposition}\label{Proconvergence}
Let $|\gamma| \leq N-2$ with $\gamma_s = 0$. Then, for every $Z^\beta \in \lambda^{|\beta|}$, we have
$$ \forall \, (t,x) \in [0,\infty) \times \R_x^2, \qquad \left|e^t \nabla_x Z^\gamma \phi \left(t, x \right) - \nabla_u \phi_{\mathrm{asymp}} \big[ Z^\gamma_u Q_\infty \big]\left( \frac{x}{e^t} \right) \right| \lesssim \epsilon \, \langle t \rangle^{N+5}e^{-t}.$$
\end{proposition} 

\begin{proof}
Let $Z^\gamma \in \lambda^{|\beta|}$ and assume for simplicity that $Z^\gamma$ does not contain the vector field $L$. Then, by the commutation formula of Propositions \ref{lemma_commuted_poisson_equation} and \ref{Proasymp},
$$ \Delta Z^\gamma \phi = \rho \big(Z^\gamma f \big), \qquad \Delta  \phi_{\mathrm{asymp}}[Z_u^\gamma Q_\infty] = Z^\gamma_u Q_\infty=Q^\gamma_\infty.$$
Now, by a change of variables, we have
$$ e^t \nabla_x \Phi\left(t,x \right) = \nabla_u  \phi_{\mathrm{asymp}}[Z^\gamma_u Q_\infty]  \left( \frac{x}{e^t} \right), \qquad \Delta \Phi(t,x)=\frac{1}{e^{2t}} Q_\infty^\gamma \left( \frac{x}{e^t} \right).$$
Consequently,
$$ \left|e^t \nabla_x Z^\gamma \phi \left(t, x \right) - \nabla_u \phi_{\mathrm{asymp}} \big[Z^\gamma_u Q_\infty \big]\left( \frac{x}{e^t} \right) \right| =e^t\left| \nabla_x \big( Z^\gamma \phi-\Phi \big) \right|\left(t, x \right). $$
To prove the result, it suffices to control $\nabla_x \Psi$, where
$$ \Delta \Psi (t,x) = \rho \big( Z^\gamma f \big)-\frac{1}{e^{2t}}Q_\infty^\gamma \left( \frac{x}{e^t} \right).$$ The proposition follows by using Proposition \ref{prop_self_similar_spatial_density} and the integral estimate \eqref{remark_uniform_integral_bound_kernel_convolution_duan}.
\end{proof}

Let us recall that $2u=x+v$ and $2s=x-v$. We investigate the convergence of $e^t \nabla_x  \phi \left(t, e^t u+e^{-t}s \right)$ along the linear spatial characteristics given by $t \mapsto (t,e^tu+e^{-t}s)$.

\begin{proposition}\label{cor_force_field_along_lineartrajectories}
For all $(t,s,u) \in [0,\infty) \times \R_s^2\times \R_u^2$ and every $|\gamma|\leq N-2$, we have
$$  \left|e^t \nabla_x Z^\gamma \phi \left(t, e^t u+e^{-t}s \right) - \nabla_u \phi_{\mathrm{asymp}} \big[Z^\gamma_u Q_\infty \big]\left( u  \right) \right| \lesssim \epsilon \, \langle e^{-t}s \rangle \, (1+t)^{N+5} e^{-t}.$$
\end{proposition}

\begin{proof}
Fix $(t,s,u) \in [0,\infty) \times \R_s^2\times \R_u^2$ and apply the mean value theorem. We obtain
$$
 \left| e^t \nabla_x Z^\gamma \phi \left(t, e^tu+e^{-t}s \right)-e^t \nabla_x Z^\gamma \phi (t, e^tu ) \right|  \leq |s|\sup_{y \in \R^3} \left| \nabla^2 Z^\gamma \phi \right|(t,y) .$$
Use then the decay of the force field $|\nabla^2 Z^\gamma \phi|(t,x)\lesssim \epsilon e^{-3t}$ and Proposition \ref{Proconvergence}.
\end{proof}

Proposition \ref{cor_force_field_along_lineartrajectories} allows us to deduce the following corollary, which it will be useful when we will consider improved commutators.

\begin{corollary}\label{Cor_conv_formodifiedvector}
For all $(t,x,v) \in [0,\infty) \times \R_x^2 \times \R^2_v$ and every $|\gamma| \leq N-2$, we have, with $2u=x+v$,
$$  \left|e^t \nabla_x Z^\gamma \phi \left(t, x \right) - \nabla_u \phi_{\mathrm{asymp}} \big[Z^\gamma_u Q_\infty \big]\left( e^{-t}u  \right) \right| \lesssim \epsilon \,  (1+t)^{N+6} e^{-t} \mathbf{z}_{\mathrm{mod}}(t,x,v).$$
\end{corollary}

\begin{proof}
We write $x=e^t e^{-t}u-e^{-t} e^t s$, where $2u=x+v$ and $2s=x-v$, and we apply Proposition \ref{cor_force_field_along_lineartrajectories}, with a slight abuse of notations, to $(e^ts,e^{-t}u)$. It remains to apply Lemma \ref{lemma_varphi}, which implies $\langle s \rangle \leq \langle e^t(x-v) \rangle \lesssim (1+t) \mathbf{z}_{\mathrm{mod}}(t,x,v).$
\end{proof}

We deduce from the Proposition \ref{cor_force_field_along_lineartrajectories} a uniform bound on $\nabla_u \phi_{\mathrm{asymp}} \big[ Z^\gamma_u Q_\infty \big]$. 

\begin{proposition}\label{prop_improved_estimate_self_similar_force_field_profile}
For every $|\gamma|\leq N-2$, we have $\|\nabla_u \phi_{\mathrm{asymp}} \big[ Z^\gamma_u Q_\infty \big]\|_{L^{\infty}_u}\lesssim \e$. 
\end{proposition}

\begin{proof}
This is implied by the estimate $|e^t \nabla_x Z^\gamma \phi (t, e^tu )|\lesssim \e$ obtained in Proposition \ref{proposition_estimate_phi}. 
\end{proof}

For later use, we prove that the structure of the asymptotic force field is preserved by differentiation.

\begin{corollary}\label{corollary_structure_asymp}
For every $|\gamma|\leq N-2$, we set $$\Delta_{Z^{\gamma}, i}(t,x,v):=e^t\partial_{x^i}Z^{\gamma}\phi (t,x)-\partial_{u^i} \phi_{\mathrm{asymp}} \big[ Z^\gamma_u Q_\infty \big]\Big(\frac{x+v}{2}\Big).$$ For every $|\gamma|\leq N-3$, there holds 
\begin{align}
U_j(\Delta_{Z^{\gamma},i})&=\Delta_{U_jZ^{\gamma},i}, 
\qquad L(\Delta_{Z^{\gamma},i})=\Delta_{LZ^{\gamma},i}-\Delta_{Z^{\gamma},i},\label{preservation_structure_asymp_force_easy} \\
R_{12}(\Delta_{Z^{\gamma},i})&=\Delta_{R_{12}Z^{\gamma},i}-\delta^1_i \Delta_{Z^{\gamma},2}+\delta^2_i \Delta_{Z^{\gamma},1}.
\end{align}
\end{corollary}

\begin{proof}
First, we compute the derivatives 
\begin{align*}
U_j(e^t\partial_{x^i}Z^{\gamma}\phi (t,x))&=e^t\partial_{x^i}U_iZ^{\gamma}\phi (t,x),\\
L(e^t\partial_{x^i}Z^{\gamma}\phi (t,x))&=e^t\partial_{x^i}LZ^{\gamma}\phi (t,x)-e^t\partial_{x^i}Z^{\gamma}\phi (t,x). 
\end{align*}
For the angular term, we compute $$R_{12}(e^t\partial_{x^i}Z^{\gamma}\phi (t,x))=e^t\partial_{x^i}R_{12}Z^{\gamma}\phi (t,x)-\delta_{i}^1e^t\partial_{x^2}Z^{\gamma}\phi (t,x)+\delta_{i}^2e^t\partial_{x^1}Z^{\gamma}\phi (t,x).$$
\end{proof}

\subsection{Convergence of the distribution along the modified characteristics}

Motivated by \cite[Section 6]{VV23} and by Proposition \ref{cor_force_field_along_trajectories}, we modify the linear stable characteristics $t\mapsto e^{-t}s$ as follows.

\begin{definition}
Let $(s,u)\in \R^2_s\times\R^2_u$. We define the \emph{modified stable characteristic} $S_{\C}(\cdot,s,u)\!:t \mapsto e^{-t}(s+\C(t,u))$ to be the trajectory given by
\begin{align*} S_{\C}(t,s,u):=e^{-t}\Big(s+\frac{1}{2}\mu t\nabla_u \phi_{\mathrm{asymp}} \big[Q_\infty \big](u)\Big). 
\end{align*} 
We also define the \emph{correction of the stable characteristic} as $\C(t,u)=\frac{1}{2}\mu t\nabla_{u} \phi_{\mathrm{asymp}} \big[Q_\infty \big](u)$. Observe that the components $\C^i(t,u)$ of the correction term $\C(t,u)$ verify $$\forall t \in [0,\infty),\qquad  |\C^i|(t,u) \lesssim (1+t) |\partial_{u^i} \phi_{\mathrm{asymp}} \big[Q_\infty \big]|\lesssim \e(1+t).$$
\end{definition}
We now need to adapt the result of Proposition \ref{cor_force_field_along_lineartrajectories} for the modified trajectories.
\begin{proposition}\label{cor_force_field_along_trajectories}
For every $(t,x,v) \in [0,\infty) \times \R_x^2\times \R_v^2$ and every $|\gamma|\leq N-2$, we have
$$  \left|e^t \nabla_x Z^\gamma \phi \big(t, e^tu+S_{\mathscr{C}}(t,s,u) \big) - \nabla_u \phi_{\mathrm{asymp}} \big[Z^\gamma_u Q_\infty \big]\left( u  \right) \right| \lesssim \epsilon \frac{(1+t)^{N+6}}{e^t} \mathbf{z}_{\mathrm{mod}}(t,x,v).$$
\end{proposition}
\begin{proof}
It suffices to apply Proposition \ref{cor_force_field_along_lineartrajectories} to $(e^t S_{\mathscr{C}}(t,s,u),u)$ and to use that 
$$\langle S_{\mathscr{C}}(t,s,u) \rangle \lesssim \langle e^{-t}s \rangle+\langle e^{-t} \C (t,u) \rangle \lesssim \langle e^{-t}s \rangle \lesssim (1+t) \mathbf{z}_{\mathrm{mod}}(t,x,v),$$
where, in the last step, we applied Lemma \ref{lemma_varphi}.
\end{proof}

We now estimate the time derivative of a distribution function $g(t,S_{\C}(t,s,u),e^tu)$ evaluated along the modified characteristics.

\begin{proposition}\label{ProTphih}
Let $g: [0,\infty)\times\R^2_s\times \R^2_u\to [0,\infty)$ be a sufficiently regular distribution function, and set $h(t,s,u):=g(t,S_{\C}(t,s,u),e^tu)$. Then, for all $(t,s,u)\in [0,\infty)\times \R^2_s\times \R^2_u$ we have, $$|\partial_t h(t,s,u)|\lesssim |\mathbf{T}_{\phi}(g)|(t,S_{\C}(t,s,u),e^tu)+\e \frac{(1+t)^{N+6}}{e^{t}}\sum_{Z\in \lambda} |\mathbf{z}_{\mathrm{mod}} Z g|(t,S_{\C},e^tu).$$
\end{proposition}

\begin{proof}
We have, for all $(t,s,u)\in[0,\infty)\times \R^2_s\times \R^2_u$, 
\begin{align*}
|\partial_t h(t,s,u)|&\leq\Big|(\partial_t g +u\cdot \nabla_u g-s\cdot \nabla_s g-\frac{\mu}{2}\nabla_x\phi\cdot \nabla_ug+\frac{\mu}{2}\nabla_x\phi\cdot \nabla_sg)(t,S_{\C}(t,s,u),e^tu)\Big|\\
&\quad+\frac{1}{2}\Big|(\nabla_x\phi\cdot \nabla_ug+(\nabla_u \phi_{\mathrm{asymp}} \big[Q_\infty \big] -e^{t}\nabla_x\phi)\cdot e^{-t}\nabla_sg)(t,S_{\C}(t,s,u),e^tu)\Big|\\
&\lesssim |\mathbf{T}_{\phi}(g)|(t,S_{\C}(t,s,u),e^tu)+e^{-2t}|e^t\nabla_x\phi\cdot e^t\nabla_ug|(t,S_{\C}(t,s,u),e^tu)\\
&\qquad + \big|e^{t}\nabla_x\phi-\nabla_u \phi_{\mathrm{asymp}} \big[Q_\infty \big]\big| |e^{-t}\nabla_s g |(t,S_{\C}(t,s,u),e^tu).
\end{align*}
Finally, we use Proposition \ref{cor_force_field_along_trajectories} to obtain $$|\partial_t h(t,s,u)|\lesssim |\mathbf{T}_{\phi}(g)|(t,S_{\C}(t,s,u),e^tu)+\e (1+t)^{N+6}e^{-t}\sum_{Z\in \lambda} |\mathbf{z}_{\mathrm{mod}} Zg|(t,S_{\C},e^tu).$$
\end{proof}

By applying this result to the Vlasov field $\bar{f}$, we obtain the existence of a distribution $\bar{f}_{\infty} \in L^{\infty}_{s,u}$ such that $f(t,S_{\C},u)\to \bar{f}_{\infty}(s,u)$ as $t\to \infty$. Applying this argument to $\partial_{s^i}\bar{f}$, we deduce that $\bar{f}_{\infty}$ is $C^1$ with respect to the stable variable $s$ if $N \geq 3$. Obtaining the regularity with respect to the unstable variable $u$ requires a more thorough analysis. 

\subsection{Modified commutators}

We assume, for the rest of this article, that $N \geq 3$. Let $Z\in \lambda\setminus \{ S_1,S_2\}$ be a vector field. Contrary to the case of the stable vector fields, the error term $[\mathbf{T}_{\phi},Z]f$ does not decay sufficiently fast in order to prove a convergence result for $Zf$, even along the modified characteristics. Recall from Lemma \ref{LemComfirstorder} that 
\begin{equation}\label{eq:commutator}
[\T_{\phi},Z]=\mu \sum_{k=1}^2 \partial_{x^k}(Z\phi+c_Z\phi)\partial_{v^k},
\end{equation}
 where $c_Z=-2$ if $Z=L$, otherwise, $c_Z=0$. Rewriting $\partial_{v^j}$ in terms of the stable and unstable vector fields, and estimating the force field through Lemma \ref{proposition_estimate_phi}, we have 
\begin{equation}\label{estimate_pre_def_modified_vector_fields_asymptotic}
\Big|\mathbf{T}_{\phi}(Zf)+\frac{\mu}{2}\sum_{k=1}^2 e^{t}\partial_{x^k}(Z\phi+c_Z\phi)e^{-t}\partial_{s^k}\Big|\lesssim \frac{\e}{e^{2t}}\sum_{Z\in\lambda}|Zf|.
\end{equation}
In view of Proposition \ref{prop_point_bound_deriv_distrib}, the right hand side is bounded by $\e(1+t)e^{-2t}$ and then it belongs to $L^1_t L^{\infty}_{x,v}$. On the other hand, if $\nabla_u \phi_{\mathrm{asymp}} \big[ ZQ_\infty \big]$ does not vanish, the decay rate of $|e^t\nabla_xZ\phi(t,x)|$ along the particle trajectories is not time integrable. For this reason, we modify the linear commutator $Z$ in a similar fashion than the modification made for the characteristic flow. Motivated by Corollary \ref{Cor_conv_formodifiedvector} and \eqref{estimate_pre_def_modified_vector_fields_asymptotic}, we introduce the following set of modified vector fields.

\begin{definition}
For every $Z\in \lambda\setminus \{ S_1,S_2\}$, we define the \emph{asymptotic modified vector fields} $Z^{\mathrm{mod}}$ as 
\begin{align*}
Z^{\mathrm{mod}}:=Z+\frac{1}{2}\mu t\sum_{k=1}^2\partial_{u^k} \phi_{\mathrm{asymp}} \big[ Z^uQ_\infty \big]\left( \frac{u}{e^t} \right)S_k+c_Z \partial_{u^k} \phi_{\mathrm{asymp}} \big[ Q_\infty \big]\left( \frac{u}{e^t} \right)S_k.
\end{align*}
We also define the \emph{correction coefficients of the asymptotic modified vector fields} as $$\C^k_Z(t,u)=\frac{1}{2}\mu t\partial_{u^k} \phi_{\mathrm{asymp}} \big[ Z^uQ_\infty \big]\big( e^{-t}u \big)+c_Z\frac{1}{2}\mu t\partial_{u^k} \phi_{\mathrm{asymp}} \big[ Q_\infty \big]\big( e^{-t}u \big),$$ so that $Z^{\mathrm{mod}}=Z+\C^1_Z(t,u)S_1+\C^2_Z(t,u)S_2$. 
\end{definition}

We have the improved commutation relations.

\begin{proposition}\label{Promodfirstorder}
Let $Z\in \lambda \setminus \{ S_1, \, S_2 \}$. Then, for every $(t,x,v) \in [0,\infty) \times \R^2_x \times \R^2_v$ we have 
\begin{align*}
2[\T_{\phi},Z^{\mathrm{mod}}] & = -\mu \left(e^t  \nabla_x Z \phi  (t,x) - \partial_{u^k} \phi_{\mathrm{asymp}} \big[ Z^uQ_\infty \big]\left( \frac{u}{e^t} \right) \right) \cdot S +\mu e^{-t} \nabla_x Z\phi (t,x) \cdot U \\
& \quad - c_Z \mu \left(e^t  \nabla_x \phi  (t,x) - \partial_{u^k} \phi_{\mathrm{asymp}} \big[ Q_\infty \big]\left( \frac{u}{e^t} \right) \right) \cdot S +c_Z \mu e^{-t} \nabla_x \phi (t,x) \cdot U \\
& \quad +\mu \! \sum_{1 \leq k \leq 2} \! \C^k_Z (t,u) \nabla_x S_k \phi (t,x) \cdot \big( e^{-t}U-e^tS \big)-e^{-t}\nabla_x \phi(t,x) \cdot U \big(\C^k_Z \big)(t,u) S_k.
\end{align*}
\end{proposition}

\begin{proof}
In view of the commutation relation \eqref{eq:commutator},
\begin{align*}
[\T_\phi, Z^{\mathrm{mod}}]&=[\T_F,Z]+ \sum_{1 \le k \le 2}\C^k_{Z} \, [\T_\phi,S_k] +\T_\phi \left( \C^k_{Z} \right)S_k\\
& =\mu \nabla_x Z \phi \cdot \nabla_v+c_Z \mu \nabla_x  \phi \cdot \nabla_v +\sum_{1 \leq k \leq 2} \mu \C^k_Z \nabla_x S_k \phi \cdot \nabla_v +\T_\phi \left( \C^k_{Z} \right)S_k.
\end{align*}
Then, we write $2\nabla_v=e^{-t} U-e^t S$ and we compute, using $\T_0(e^{-t}u)=0$,
$$
 \T_\phi \left( \C^k_Z \right) =\frac{1}{2}\mu \partial_{u^k} \phi_{\mathrm{asymp}} \big[ (Z^u+c_Z)Q_\infty \big]\left( \frac{u}{e^t} \right)-\mu \nabla_x \phi \cdot \nabla_v \C^k_Z (t,u). $$
\end{proof}

We now state the higher order commutation formula. For this, it will be convenient to denote by $P_{p,q}(\C)$ to any quantity of the form
$$ \prod_{1 \leq i \leq p} Z^{\gamma^i}\big( \C^{k_i}_{Z^{\ell_i}} \big), \quad p  \!\in \! \mathbb{N}, \; \; \, 0 \leq q \leq N-3, \; \; \, k_i \!\in \! \{1,2 \}, \; \; \, 1 \leq \ell_i \leq 4, \; \; \, \sum_{1 \leq i \leq p} |\gamma^i| =q, \; \;\, \gamma^i_u=|\gamma^i| .$$
The last condition means that $Z^{\gamma^i}$ is only composed by unstable vector fields $U_i$, $L$ and $R_{12}$. If $p=0$, then, we set the convention that $P_{0,q}(\C)=1$. Note that Proposition \ref{Proasymp} implies
\begin{align} \nonumber
 \big| P_{p,q} (\C) \big|(t,u) &= \Big|\prod_{1 \leq i \leq p}\partial_{u^1}^{\gamma_1^i} Z_u^{\gamma^i} \big( \C^{k_i}_{Z^{\ell_i}} \big)\Big|(t,u)   \lesssim  \sum_{|\gamma| \leq q+1} \big|(1+t) \phi_{\mathrm{asymp}}\big[ Z^\gamma_u Q_\infty \big](u) \big|^p \\ & \lesssim  \epsilon^p (1+t)^p. \label{eq:boundPpq}
 \end{align}
Note further that the functions $P_{p,q} (\C)$ can be used in order to express a differential operator $Z^{\mathrm{mod},\beta}$, for $|\beta| \leq N-2$, in terms of differential operators $Z^\kappa$. We have, as $S_k P_{p,q}(\C)=0$, that
\begin{equation}\label{eq:modtonormal}
 Z^{\mathrm{mod}, \beta} = \sum_{q+|\kappa| \leq |\beta|, } \sum_{p \leq |\beta|, \, q \leq |\beta| -1} D_{\kappa, p , q}^{\beta} \,  P_{p,q}(\C) Z^{\kappa}, \qquad D^\beta_{\kappa,p,q} \in \mathbb{Z}   .
 \end{equation}
Because of regularity issues on the coefficients $\C_Z$, we are only able to commute the Vlasov equation by $Z^{\mathrm{mod},\beta}$ where $|\beta|\leq N-3$.

\begin{proposition}\label{Procommod}
Let $Z^{\mathrm{mod},\beta}\in\lambda_{\mathrm{mod}}^{|\beta|}$ with $|\beta| \leq N-3$. Then, we can write the commutator $[\T_{\phi},Z^{\mathrm{mod},\beta}]$ as a linear combination of the following two types of terms, 
\begin{align*}
&  \bullet \; \, P_{p,q} (\C) \left(e^t  \partial_{x^i} Z^\gamma \phi  (t,x) - \partial_{u^i} \phi_{\mathrm{asymp}} \big[ Z^\gamma_uQ_\infty \big]\left( \frac{u}{e^t} \right) \right) Z^\kappa  , \\
&  \bullet \; \, P_{p,q} (\C) e^{-t} \partial_{x^i} Z^\gamma \phi (t,x) Z^\kappa  ,
\end{align*}
where 
$$ |\gamma|+|\kappa| \leq |\beta|+1, \qquad |\gamma|, \, |\kappa| \leq |\beta|, \qquad p \leq |\beta|, \qquad q \leq |\beta|, \qquad i \in \{1,2 \}.$$
\end{proposition}

\begin{proof}
For the case $|\beta|=1$, apply Proposition \ref{Promodfirstorder} and note, in view of $S_k \phi = e^{-2t} U_k \phi$, 
$$\nabla_x S_k \phi (t,x) \cdot e^tS = e^{-t} \nabla U_k \phi (t,x) \cdot S .$$
Let $1 \leq n \leq N-4$ such that the result holds for all multi-indices $|\beta| \leq n$. Fix then $|\beta|=n$, $Z^{\mathrm{mod}} \in \lambda_{\mathrm{mod}}$ and write
\begin{equation}\label{equation:forcomod}
[\T_\phi , Z^{\mathrm{mod}} Z^{\mathrm{mod} , \beta}]= [\T_\phi , Z^{\mathrm{mod}}]Z^{\mathrm{mod}, \beta}+ Z^{\mathrm{mod}}[\T_\phi ,  Z^{\mathrm{mod} , \beta}] .
\end{equation} 
We will make use several times of the following properties. For any $Z \in \lambda$, $p \in \mathbb{N}$ and $q \leq N-4$, $ZP_{p,q}(\C)$ can be written as a linear combination of terms of the form $P_{p,q+1}(\C)$.

Combining the first order commutation formula with \eqref{eq:modtonormal}, which allows us to rewrite $Z^{\mathrm{mod},\beta}$ in terms of the linear commutators as well as $P_{p,q}(\C)$, and this last property, one gets that the first term on the right hand side of \eqref{equation:forcomod} has the expected form.

Next, by the induction hypothesis, we can write $Z^{\mathrm{mod}}[\T_\phi ,  Z^{\mathrm{mod} , \beta}]$ as a linear combination of terms of the form $\mathcal{T}_1[Z^{\mathrm{mod}}] $ and $\mathcal{T}_2[Z^{\mathrm{mod}}] $, where, for a vector field $X$,
\begin{align*}
\mathcal{T}_1[X] & := X \Big[ P_{p,q} (\C) \left(e^t  \partial_{x^i} Z^\gamma \phi  (t,x) - \partial_{u^i} \phi_{\mathrm{asymp}} \big[ Z^\gamma_uQ_\infty \big]\left( \frac{u}{e^t} \right) \right) Z^\kappa \Big] , \\
\mathcal{T}_2[X] & := X \Big[   P_{p,q} (\C) e^{-t} \partial_{x^i} Z^\gamma \phi (t,x) Z^\kappa \Big],
\end{align*}
and $|\gamma|+|\kappa| \leq |\beta|+1$, $|\gamma| \leq |\beta|$, $|\kappa| \leq |\beta|$, $p \leq |\beta|$, $q \leq |\beta|+1$ as well as $i \in \{1,2 \}$. If $X=S_k$, for $k \in \{1,2 \}$, then, using $S_k(u)=S_k(t)=0$, we get
$$ \mathcal{T}_1[S_k]= e^t S_k \big(\partial_{x^i}Z^\gamma \phi  (t,x) \big) Z^\kappa + P_{p,q} (\C) \left(e^t  \partial_{x^i} Z^\gamma \phi  (t,x) - \partial_{u^i} \phi_{\mathrm{asymp}} \big[ Z^\gamma_uQ_\infty \big]\left( \frac{u}{e^t} \right) \right) S_k Z^\kappa.$$
Clearly, the second term on the right hand side has the expected form. For the first one, use that $e^tS_k \psi = e^{-t}U_k \psi$ for any function $\psi (t,x)$ and then $[U_k,\partial_{x^i}]=0$. For $\mathcal{T}_2[S_k]$, using the same arguments, we get
$$ \mathcal{T}_2[S_k] =  P_{p,q} (\C) e^{-t} \partial_{x^i} S_k Z^\gamma \phi (t,x) Z^\kappa +P_{p,q} (\C) e^{-t} \partial_{x^i}  Z^\gamma \phi (t,x) S_k Z^\kappa,$$
which concludes the case $Z^{\mathrm{mod}}=S_k$. Otherwise $Z^{\mathrm{mod}} \in \lambda_{\mathrm{mod}} \setminus \{ S_1 , S_2 \}$, so that
$$ \mathcal{T}_k \big[ Z^{\mathrm{mod}} \big]= \mathcal{T}_k \big[ Z \big]+\C^1_Z \mathcal{T}_k \big[ S_1 \big]+\C^2_Z \mathcal{T}_k \big[ S_2 \big], \qquad k \in \{1 , 2 \}.$$
In the view of the analysis of the case $Z^{\mathrm{mod}}=S_k$, the last two terms on the right hand side has the expected form. For the first one, we have
$$ \mathcal{T}_1 \big[ Z \big]= Z \big( P_{p,q} (\C) \big) \Delta^i_{Z^\gamma} Z^\kappa +  P_{p,q} (\C) Z \big( \Delta^i_{Z^\gamma} \big) Z^\kappa+P_{p,q} (\C)  \Delta^i_{Z^\gamma} Z Z^\kappa .$$
The first and the last terms on the right hand side have the required form. The same holds true for the second one according to Corollary \ref{corollary_structure_asymp}. Finally, 
$$ \mathcal{T}_2 \big[ Z \big]= Z \big( P_{p,q} (\C) \big) e^{-t} \partial_{x^i} Z^\gamma \phi Z^\kappa +  P_{p,q} (\C) Z \big( e^{-t} \partial_{x^i} Z^\gamma \phi  \big) Z^\kappa+P_{p,q} (\C)e^{-t} \partial_{x^i} Z^\gamma \phi  ZZ^\kappa $$
and since $[e^{-t} \partial_{x^i},Z]\in \{0, \pm e^{-t} \partial_{x^1}, \pm e^{-t} \partial_{x^2} \}$, all the terms on the right hand side have the expected form.
\end{proof}

We now control the two type of error terms in Proposition \ref{Procommod}. As a result, we prove a uniform boundedness statement for the derivatives $Z^{\mathrm{mod},\beta}f$.

\begin{proposition}\label{ProboundVlasovmod}
For any $|\beta|\leq N-3$, we have 
$$ \forall \, (t,x,v) \in [0,\infty) \times \R^2_x \times \R^2_v, \qquad \big|\T_{\phi} \big(Z^{\mathrm{mod},\beta}f \big) \big|(t,x,v) \lesssim \epsilon \frac{\langle t \rangle^{2N+4}}{e^t}\sum_{|\kappa|\leq |\beta|} \big|\mathbf{z}_{\mathrm{mod}} Z^\kappa f \big|(t,x,v).$$
\end{proposition}

\begin{proof}
Fix $(t,x,v) \in [0,\infty) \times \R^2_x \times \R^2_v$. Combining Proposition \ref{Procommod} with the bound \eqref{eq:boundPpq} on $P_{p,q}(\C)$, we obtain that $\T_{\phi} (Z^{\mathrm{mod}}f )$ is bounded by
$$ \langle t \rangle^{N-2} \! \sum_{|\gamma|, \, |\kappa| \leq N-3}  \!\Big[ \Big|e^t  \partial_{x^i} Z^\gamma \phi  (t,x) - \partial_{u^i} \phi_{\mathrm{asymp}} \big[ Z^\gamma_uQ_\infty \big] \! \left( \frac{u}{e^t} \right) \! \Big| +\frac{1}{e^t}\big| \nabla_x Z^\gamma \phi \big|(t,x) \Big]\big| Z^\kappa f \big| (t,x,v). $$
It remains to estimate the force field through Proposition \ref{proposition_estimate_phi}, so that $|\nabla_x Z^\gamma \phi|(t,x) \lesssim \epsilon e^{-t}$, and to apply Corollary \ref{Cor_conv_formodifiedvector}.
\end{proof}

We state the next corollary for completeness, even if we will not need it in order to prove that $\bar{f}_\infty \in C^{N-3}(\R^2_u \times \R^2_s)$. 

\begin{corollary}\label{cor_sharp_Linfty_estimates_modif_commut}
For any $|\beta|\leq N-3$, there holds
 $$ \forall \, (t,x,v) \in [0,\infty) \times \R^2_x \times \R^2_v, \qquad  \big( \langle e^{-t}(x+v) \rangle^{M-1} +\mathbf{z}_{\mathrm{mod}}^{M-1}(t,x,v)\big) \big| Z^{\mathrm{mod},\beta}f\big|(t,x,v)\leq 2\e.$$
\end{corollary}

Corollary \ref{cor_sharp_Linfty_estimates_modif_commut} is a direct consequence of Proposition \ref{ProboundVlasovmod}, the property $\T_\phi(\mathbf{z}_{\mathrm{mod}})=0$, the bound $|\T_\phi (e^{-t}(x+v))| \lesssim \epsilon \, e^{-2t}$ obtained in \eqref{eq:step1}, and the $L^\infty_{x,v}$ estimates in Proposition \ref{prop_point_bound_deriv_distrib}.

\subsection{Regularity of the scattering state}

In order to prove that $\bar{f}_{\infty}$ is differentiable with respect to $u$, we need to compute the first order derivatives of the correction terms in the modified characteristics. We will bound their higher order derivatives in the following.

\begin{lemma}\label{Lemderivmodifcoeff}
Let $i,k\in \{1,2\}$. Then, for every $(t,u)\in[0,\infty)\times\R^2_u$, we have $$\partial_{u^k}\C^i(t,u)=e^{-t}\C^i_{U_k}(t,e^t u).$$ 
\end{lemma}

\begin{proof}
We have $$\partial_{u^k}\C^i(t,u)=\partial_{u^k}\Big(\frac{\mu }{2e^t} t\partial_{u^i} \phi_{\mathrm{asymp}} \big[Q_\infty \big](u)\Big)= \frac{\mu }{2e^t} t\partial_{u^i} \phi_{\mathrm{asymp}} \big[\partial_{u^k}Q_\infty \big](u)=e^{-t}\C^i_{U_k}(t,e^t u).$$ 
\end{proof}

Before arriving to the main result of the paper, we obtain a useful lemma to estimate the derivatives of the profile $f(t,S_{\C}(t,s,u), e^tu )$.

\begin{lemma}\label{Lemderivh}
Let $h \in C^1( [0,\infty) \times \R^2_s \times \R^2_u , \R)$. Then, for any $\{ 1, 2 \}$,
\begin{align*}
 \partial_{s^i} \big[ h(t,S_{\C}(t,s,u), e^tu ) \big] &= \big[ S_i h \big] (t,S_{\C}(t,s,u), e^t u ), \\
  \partial_{u^i} \big[ h(t,S_{\C}(t,s,u), e^tu ) \big] &= \big[ U_i^{\mathrm{mod}} h \big] (t,S_{\C}(t,s,u), e^t u ).
  \end{align*}
\end{lemma}
\begin{proof}
For the first identity, simply use $\partial_{s^i} S^j_{\C} = e^{-t} \delta^j_i$ as well as $e^{-t} \partial_{s^i}=S_i$. Next, since $\partial_{u^i} S^j_{\C}(t,u)=\partial_{u^i} \C^j (t,u) = e^{-t}\C^j_{U_i}(t,e^tu)$ according to Lemma \ref{Lemderivmodifcoeff}, we have
\begin{align*}
  \partial_{u^i} \big[ h(t,S{\C}(t,s,u),e^tu ) \big]= & e^{-t}\C_{U_i} (t,e^tu) \cdot \big[\nabla_s h \big](t,S_{\C}(t,s,u),e^t u ) \\
  &  +e^t\big[\partial_{u^i} h \big](t,S_{\C}(t,s,u),e^t u ).
  \end{align*}
It remains to use $e^{-t} \partial_{s^j}=S_j$ as well as $e^t \partial_{u^j}= U_j$.
\end{proof}

We are now able to prove the main modified scattering result of the paper. 

\begin{proposition}\label{prop_mod_scattering_statem_high_reg_proof}
There exists a distribution function $\bar{f}_{\infty}\in C^{N-3}(\R^2_s\times\R^2_u)$ such that for every $|\kappa|+|\xi|\leq N-3$ and all $(t,s,u) \in [0,\infty ) \times \R^2_s \times \R^2_u$, we have $$ \langle u \rangle^M \, \langle s \rangle^{M-1}\Big| \partial_u^{\xi}\partial_s^{\kappa}\big( \bar{f}(t,S_\C(t,s,u),e^t u) \big)-\partial_u^{\xi}\partial_s^{\kappa}\bar{f}_{\infty}(s,u)\Big|\lesssim \e\frac{(1+t)^{3N+M+1}}{e^t}.$$ In particular, as $M \geq 6$, we have $\partial_u^{\xi}\partial_s^{\kappa}\bar{f}_{\infty}\in L^1_{s,u}$.
\end{proposition}

\begin{proof}
Let $h(t,s,u):=\bar{f}(t,S_\C(t,s,u),e^t u)$. Iterating Lemma \ref{Lemderivh}, we have 
$$ \partial_u^{\xi}\partial_s^{\kappa} h(t,s,u) = \big[ U^{\mathrm{mod}, \, \xi} S^\kappa \bar{f} \big](t,S_\C(t,s,u),e^t u).$$
We then get, by applying Proposition \ref{ProTphih} to $\partial_u^{\xi}\partial_s^{\kappa} h$,
\begin{align*}
|\partial_t \partial_u^{\xi}\partial_s^{\kappa} h|(t,s,u) \lesssim & \,  \big| \mathbf{T}_{\phi} \big( U^{\mathrm{mod}, \, \xi} S^\kappa \bar{f} \big) \big|(t,S_{\C}(t,s,u),e^tu) \\
& +\e (1+t)^{N+6}e^{-t}\sum_{Z\in \lambda} |\mathbf{z}_{\mathrm{mod}} Z U^{\mathrm{mod}, \, \xi} S^\kappa\bar{f}|(t,S_{\C}(t,s,u),e^tu).
\end{align*}
The first term is controlled in Proposition \ref{ProboundVlasovmod}. In order to bound the second term on the right hand side, we express $U^{\mathrm{mod}, \, \xi}$ in terms of $Z^\alpha$ and $P_{p,q}(\C)$ through \eqref{eq:modtonormal}. Since $|ZP_{p,q}(\C)| \lesssim (1+t)^p$ by \eqref{eq:boundPpq}, we get
$$ |\partial_t \partial_u^{\xi}\partial_s^{\kappa} h|(t,s,u) \lesssim \e (1+t)^{2N+4}e^{-t}\sum_{| \alpha| \leq N-2} |\mathbf{z}_{\mathrm{mod}} Z^\alpha \bar{f}|(t,S_{\C}(t,s,u),e^tu).$$
Note now that Proposition \ref{prop_improved_estimate_self_similar_force_field_profile} and Lemma \ref{lemma_varphi} provide
\begin{align*}
|s| &\leq  e^{-t}|S_{\C}|(t,s,u)+e^{-t}|\C|(t,s,u) \lesssim e^{-t}|S_{\C}|(t,s,u)+\epsilon, \\
 e^{-t}|s'| &\leq \mathbf{z}_{\mathrm{mod}}(t,e^{-t}s',u') +|\overline{\varphi}|(t,e^{-t}s',u') \leq 2(1+t)\mathbf{z}_{\mathrm{mod}}(t,e^{-t}s',u').
\end{align*}
We then deduce that 
\begin{align*}
 \langle u \rangle^{M} \, \langle s \rangle^{M-1}|\partial_t \partial_u^{\xi}\partial_s^{\kappa} h|(t,s,u) & \lesssim \e \langle t \rangle^{2N+M+3}e^{-t}\sum_{| \alpha| \leq N-2}  \langle u \rangle^M \big| \mathbf{z}_{\mathrm{mod}}^M Z^\alpha \bar{f} \big|(t,S_{\C}(t,s,u),e^tu) \\
& \lesssim \e \langle t \rangle^{3N+M+1}e^{-t},
\end{align*}
where, in the last step, we apply the $L^\infty$ estimates of Proposition \ref{prop_point_bound_deriv_distrib}. It implies that $(s,u) \mapsto \langle u \rangle^{M} \, \langle s \rangle^{M-1}h(t,s,u)$ converges in $C^{N-3}(\R^2_u \times \R^2_s)$, as $t \to + \infty$, and that the stated rate of convergence holds.
\end{proof}

\section{Asymptotic properties in terms of the scattering state}\label{section_asymp_prop_scatt_state}

In this section, we obtain several asymptotic properties of small data solutions to the Vlasov--Poisson system with the potential $\frac{-|x|^2}{2}$ in terms of the scattering state. For this purpose, we revise the estimates of the normalized stable averages and the velocity averages, to obtain Theorem \ref{thm_asymp_spatial_density_complete}. We finish the paper with the proof of Theorem \ref{thm_weak_convergence_prop_complete} and Theorem \ref{thm_conservation_laws_complete}

\subsection{Asymptotics of normalized stable averages}

In this subsection, we revise the convergence of the normalized stable averages obtained in Proposition \ref{prop_conver_unstable_averages}. For this purpose, we use the modified scattering of the distribution function. 

\begin{proposition}\label{lem_constant_expansion_rho_term_scatt_state}
For every $(s,u) \in \R^2_s \times \R^2_u$, we have $$\Big|e^{2t}\int_{\R_s^2} \bar{f}(t,s,e^tu)\mathrm{d}s- \int_{\R_s^2}\bar{f}_{\infty}(s,u)\mathrm{d}s\Big|\lesssim \e\frac{(1+t)^{N+5}}{e^{t}}.$$ In other words, the normalized stable average $Q(t,u)$ converges to $\int_{\R_s^2}\bar{f}_{\infty}(s,u)\mathrm{d}s$.
\end{proposition}

\begin{proof}
Performing the change of variables $s\mapsto y(s):=e^ts-\frac{1}{2}\mu t \nabla_u \phi_{\mathrm{asymp}}[Q_{\infty}](u)$, we have $$e^{2t}\int_{\R_s^2} \bar{f}(t,s,u)\mathrm{d}s=\int_{\R^2_y}\bar{f}(t,S_{\C}(t,y,u),e^tu)\mathrm{d}y.$$ We then deduce that $$\Big|\int_{\R^2_s}e^{2t}\bar{f}(t,s,e^tu)-\bar{f}_{\infty}(s,u)\mathrm{d}s\Big|\leq \sup_{(s,u)\in \R^2_s\times\R^2_u}\langle s,u \rangle^3|\bar{f}(t,S_{\C}(t,s,u),e^tu)-\bar{f}_{\infty}(s,u)|,$$ which, in view of $M\geq 6$ and Proposition \ref{prop_mod_scattering_statem_high_reg_proof}, implies $Q_\infty =\int_{\R^2_s} \bar{f}_\infty (s,\cdot)\mathrm{d}s$. It remains to apply Proposition \ref{prop_self_similar_spatial_density}.
\end{proof}

In particular, we obtain the following corollary when $u=0$ in Proposition \ref{lem_constant_expansion_rho_term_scatt_state}.

\begin{corollary}\label{cor_limit_mass_stable_mfld_origin}
For every $(s,u) \in \R^2_s \times \R^2_u$, we have $$\Big|e^{2t}\int_{\R_s^2} \bar{f}(t,s,u)\mathrm{d}s- \int_{\R_s^2}\bar{f}_{\infty}(s,0)\mathrm{d}s\Big|\lesssim \e\langle u \rangle \frac{(1+t)^{N+5}}{e^{t}}.$$ 
\end{corollary}

\begin{proof}
After performing the change of variables $s\mapsto y(s):=e^ts-\frac{1}{2}\mu t \nabla_u \phi_{\mathrm{asymp}}[Q_{\infty}](\frac{u}{e^t})$ and the mean value theorem, the proof is identical to the one of Lemma \ref{lem_constant_expansion_rho_term_scatt_state}
\end{proof}

\subsection{Asymptotics of velocity averages}

In this subsection, we revise the decay estimate of the spatial density performed in Proposition \ref{Protimedecay}. We use the asymptotics derived in Corollary \ref{cor_limit_mass_stable_mfld_origin} for the normalized stable averages in order to obtain the precise late-asymptotic behavior of the spatial density

\begin{proposition}
For every $(t,x)\in [0,\infty) \times \R_x^2$, the spatial density satisfies 
\begin{align*}
&\Big|e^{2t}\int_{\R^2_v}f(t,x,v)\mathrm{d}v-\int_{\R^2_y}\bar{f}_{\infty}(s,0)\mathrm{d}s\Big|\lesssim\e(1+|x|)\frac{(1+t)^{7}}{e^{t}}.
\end{align*}
In other words, the normalized spatial density $e^{2t}\rho(f)$ converges to the constant $\int_{\R^2_y}\bar{f}_{\infty}(s,0)\mathrm{d}s$.
\end{proposition}

\begin{proof}
Consider $g(t,x,v)= f(t,X_{\L}(t),V_{\L}(t))$. Applying Lemma \ref{lem_properties_linear_change_variable_trivial} and Proposition \ref{prop_point_bound_deriv_distrib} to $g$, we have 
\begin{align}
\sup_{(s,u)\in\R^2_s\times\R^2_u} \big(\langle s \rangle+\langle u \rangle)^{6}|\nabla_u\bar{g}| (t,s,u)&\lesssim \sup_{(s,u)\in\R^2_s\times\R^2_u} \big(\langle e^t s \rangle+\langle e^{-t}u \rangle)^{6} \big|Z\bar{f} \big|(t,s,u)  \nonumber\\ 
&\lesssim \e (1+t)^{7}, \label{estimate_improving_BA_spatial_density_second}
\end{align}
since we have $M \geq 6$ and $\langle e^t s \rangle \lesssim (1+t)\bar{ \mathbf{z}}_{\mathrm{mod}}(t,s,u)$. Applying Lemma \ref{lem_asympt_exp_spatial_density_estimate} to the distribution $g$ and using Lemma \ref{lem_properties_linear_change_variable_trivial}, we obtain
\begin{align*} 
\int_{\R^2_y}\bar{g}\Big(t,y,\frac{x}{e^t}\Big)\mathrm{d}y=\int_{\R_y^2}\bar{f}\Big(t,\frac{y}{e^t},x\Big)\mathrm{d}y=e^{2t}\int_{\R_s^2} \bar{f}(t,s,x)\mathrm{d}s.
\end{align*}
Finally, we conclude using Corollary \ref{cor_limit_mass_stable_mfld_origin}.
\end{proof}

\subsection{Weak convergence I: Concentration in the unstable manifold}

In this subsection, we show that $e^{2t}\bar{f}(t)$ converges weakly to the Dirac mass $(\int \bar{f}_{\infty}(s,0)\mathrm{d}s)\delta_{s=0}(s)$. Let $\varphi\in C^{\infty}_{x,v}$ be a compactly supported test function. The starting point consists in performing the change of variables $y=\frac{1}{2}e^t(x-v)$ in the integral $\int g(t,X_{\L}(-t),V_{\L}(-t))\varphi(x,v)\mathrm{d}x\mathrm{d}v$.

\begin{lemma}\label{lemma_change_variables_decay_weakconv}
Let $g: [0,\infty)\times \R^2_x\times \R^2_v\to\R$ be a sufficiently regular distribution, and let $\varphi\in C^{\infty}_{x,v}$ be a compactly supported test function. Then, for every $t\in [0,\infty)$ we have $$e^{2t}\! \int_{\R^2_x\times \R_v^2}g(t,X_{\L}(-t),V_{\L}(-t))\varphi(x,v)\mathrm{d}x\mathrm{d}v=\!\int_{\R^2_x\times \R^2_y} \bar{g}\Big(t,y,\frac{1}{e^t}\Big(x\ -\ \frac{y}{e^t}\Big)\Big)\bar{\varphi}\Big(\frac{y}{2e^t},x \ -\  \frac{y}{2e^t}\Big)\mathrm{d}x\mathrm{d}y.$$
\end{lemma}

Next, we apply the mean value theorem to obtain a technical lemma that will be used to capture the weak convergence property of $e^{2t}\bar{f}(t)$.

\begin{lemma}\label{lem_asympt_exp_spatial_density_estimate_weak_conv_first_order}
Let $g: [0,\infty)\times \R^2_x\times \R^2_v\to\R$ be a sufficiently regular distribution. Then, for every $t\in [0,\infty)$ we have 
\begin{align*}
\Big| e^{2t}\int_{\R_x^2\times\R_v^2}g(t,X_{\L}(-t),V_{\L}(-t))\varphi (x,v)\mathrm{d}v\mathrm{d}x~-&\int_{\R_x^2} \bar{\varphi}(0,x) \int_{\R_y^2}\bar{g}\Big(t,y,\frac{x}{e^t}\Big)\mathrm{d}y\mathrm{d}x\Big|\\
&\lesssim \frac{1}{e^{t}}\sup_{(s,u)\in\R^2_s\times \R^2_u}\langle s \rangle^{4}(|\nabla_{s,u}\bar{\varphi}||\bar{g}|+|\bar{\varphi}||\nabla_u\bar{g}|).
\end{align*}
\end{lemma}

\begin{proof}
Using the change of variables $y=\frac{1}{2}e^t(x-v)$ and applying the mean value theorem, we have 
\begin{align*}
\Big| \bar{g}\Big(t,y,\frac{1}{e^t}\Big(x-\frac{y}{e^t}\Big)\Big)&\bar{\varphi}\Big(\frac{y}{2e^t},x-\frac{y}{2e^t}\Big)-\bar{g}\Big(t,y,\frac{x}{e^t}\Big)\bar{\varphi}(0,x)\Big| \\
&\lesssim \frac{1}{e^{t}}\sup_{(s,u)\in\R^2_s\times \R^2_u}|s|(|\nabla_{s,u}\bar{\varphi}||\bar{g}|+|\nabla_u\bar{g}||\bar{\varphi}|).
\end{align*}
Then, the difference $$\Big|\int_{\R^2_x\times \R^2_y}\bar{g}\Big(t,y,\frac{1}{e^t}\Big(x-\frac{y}{e^t}\Big)\Big)\bar{\varphi}\Big(\frac{y}{2e^t},x-\frac{y}{2e^t}\Big)\mathrm{d}x\mathrm{d}y-\int_{\R^2_x\times \R^2_y}\bar{g}\Big(t,y,\frac{x}{e^t}\Big)\bar{\varphi}(0,x)\mathrm{d}x\mathrm{d}y\Big|$$ satisfies the corresponding time decay estimate.  
\end{proof}

We are now able to prove the main result of this subsection. 

\begin{proposition}\label{weak_conv_concentration_first_order_statem}
Let $\varphi\in C^{\infty}_{x,v}$ be a compactly supported test function. Then, the distribution $f(t)$ satisfies $$\lim_{t\to\infty}e^{2t}\int_{\R^2_s\times \R^2_u}\bar{f}(t,s,u)\bar{\varphi}(s,u)\mathrm{d}s\mathrm{d}u=\int_{\R^2_s} \bar{f}_{\infty}(s,0)\mathrm{d}s \int_{\R^2_s\times \R^2_u} \bar{\varphi}(s,u)\delta_{s=0}(s)\mathrm{d}s\mathrm{d}u.$$ In other words, the distribution $e^{2t}\bar{f}(t)$ converges weakly to $(\int \bar{f}_{\infty}(s,0)\mathrm{d}s)\delta_{s=0}(s)$.
\end{proposition}

\begin{proof}
Applying the previous lemma to the distribution $g(t,x,v):=f(t,X_{\L}(t),V_{\L}(t))$, we have 
\begin{align*}
\Big| e^{2t}\int_{\R_x^2\times\R_v^2}f(t,x,v)\varphi (x,v)\mathrm{d}v\mathrm{d}x~-&\int_{\R_x^2} \bar{\varphi}(0,x) e^{2t}\int_{\R_w^2}\bar{f}(t,w,x)\mathrm{d}w\mathrm{d}x\Big|\\
&\lesssim \frac{1}{e^{t}}\sup_{(s,u)\in\R^2_s\times \R^2_u}\langle s \rangle^{4}(|\nabla_{s,u}\bar{\varphi}||\bar{f}|+|\nabla_u\bar{f}||\bar{\varphi}|).
\end{align*}
We conclude by using the dominated convergence theorem and Corollary \ref{cor_limit_mass_stable_mfld_origin} to obtain
\begin{align*}
\lim_{t\to\infty}\int_{\R^2_x}\bar{\varphi}(0,x)e^{2t}\int_{\R^2_w}\bar{f}(t,w,x)\mathrm{d}w\mathrm{d}x=\int_{\R^2_x}\bar{\varphi}(0,x)\mathrm{d}x\int_{\R_w^2}\bar{f}_{\infty}(w,0)\mathrm{d}w.
\end{align*}
\end{proof}

\subsection{Weak convergence II: Hyperbolicity of the Hamiltonian flow}

Let $\bar{u}\in \R^2_u$. In this subsection, we show that $e^{2t}\bar{f}(t,s,u+\bar{u}e^t)$ converges weakly to the Dirac mass $(\int \bar{f}_{\infty}(s,\bar{u})\mathrm{d}s)\delta_{s=0}(s)$. The starting point consists in performing the change of variables $y=e^ts$ in the integral $\int \bar{g}(t,S_{\L}(-t),U_{\L}(-t)) \bar{\varphi}(s,u-e^t\bar{u})\mathrm{d}s\mathrm{d}u$, where $\bar{\varphi}\in C^{\infty}_{s,u}$ is a compactly supported test function. 

\begin{lemma}\label{lemma_change_variables_decay_weakconv_hyperbolic}
Let $\bar{g}: [0,\infty)\times \R^2_s\times \R^2_u\to\R$ be a regular distribution, and let $\bar{\varphi}\in C^{\infty}_{s,u}$ be a compactly supported test function. Then, for every $t\in [0,\infty)$ we have $$e^{2t}\int_{\R^2_s\times \R_u^2}\bar{g}(t,S_{\L}(-t),U_{\L}(-t))\bar{\varphi}(s,u-e^t\bar{u})\mathrm{d}s\mathrm{d}u=\int_{\R^2_y\times \R^2_u} \bar{g}\Big(t,y,\frac{u}{e^t}+\bar{u}\Big)\bar{\varphi}\Big(\frac{y}{e^t},u\Big)\mathrm{d}y\mathrm{d}u.$$
\end{lemma}

Next, we apply the mean value theorem to obtain a technical lemma that will be used to capture the weak convergence property of $e^{2t}\bar{f}(t,s,u+\bar{u}e^t)$.

\begin{lemma}\label{lem_asympt_exp_spatial_density_estimate_weak_conv_first_order_hyperbolic}
Let $\bar{g}: [0,\infty)\times \R^2_s\times \R^2_u\to\R$ be a sufficiently regular distribution. Then, for every $t\in [0,\infty)$ we have 
\begin{align*}
\Big| e^{2t}\int_{\R_s^2\times\R_u^2}\bar{g}(t,S_{\L}(-t),U_{\L}(-t))\bar{\varphi} (s,u-e^t\bar{u})\mathrm{d}s&\mathrm{d}u~-\int_{\R_u^2} \bar{\varphi}(0,u)\mathrm{d}u \int_{\R_y^2}\bar{g}(t,y,\bar{u})\mathrm{d}y\Big|\\
&\lesssim \frac{1}{e^{t}}\sup_{(s,u)\in\R^2_s\times \R^2_u}\langle s \rangle^{4}(|\nabla_{s,u}\bar{\varphi}||\bar{g}|+|\bar{\varphi}||\nabla_u\bar{g}|).
\end{align*}
\end{lemma}

\begin{proof}
Using the change of variables in Lemma \ref{lemma_change_variables_decay_weakconv_hyperbolic} and applying the mean value theorem, we have 
\begin{align*}
\Big| \bar{g}\Big(t,y,\frac{u}{e^t}+\bar{u}\Big)&\bar{\varphi}\Big(\frac{y}{e^t},u\Big)-\bar{g}(t,y,\bar{u})\bar{\varphi}(0,u)\Big| \lesssim \frac{1}{e^{t}}\sup_{(s,u)\in\R^2_s\times \R^2_u}\langle s \rangle^{4}(|\nabla_{s}\bar{\varphi}||\bar{g}|+|\bar{\varphi}||\nabla_u\bar{g}|).
\end{align*}
Then, the difference $$\Big|\int_{\R^2_y\times\R^2_u}\bar{g}\Big(t,y,\frac{u}{e^t}+\bar{u}\Big)\bar{\varphi}\Big(\frac{y}{e^t},u\Big)-\bar{g}(t,y,\bar{u})\bar{\varphi}(0,u)\mathrm{d}y\mathrm{d}u\Big|$$ satisfies the corresponding time decay estimate. 
\end{proof}

We are now able to prove the main result of this subsection. 

\begin{proposition}
Let $\bar{\varphi}\in C^{\infty}_{s,u}$ be a compactly supported test function, and let $\bar{u}\in\R^2_u$. Then, the distribution $\bar{f}(t)$ satisfies $$\lim_{t\to\infty}e^{2t}\int_{s,u}\bar{f}(t,s,u)\bar{\varphi}(s,u-e^t\bar{u})\mathrm{d}s\mathrm{d}u=\int_s \bar{f}_{\infty}(s,\bar{u})\mathrm{d}s \int_{s,u} \bar{\varphi}(s,u)\delta_{s=0}(s)\mathrm{d}s\mathrm{d}u.$$ In other words, the distribution $e^{2t}\bar{f}(t,s,u+\bar{u}e^t)$ converges weakly to $(\int \bar{f}_{\infty}(s,\bar{u})\mathrm{d}s)\delta_{s=0}(s)$.
\end{proposition}

\begin{proof}
Applying the previous lemma to the distribution $\bar{g}(t,s,u):=\bar{f}(t,S_{\L}(t),U_{\L}(t))$, we have 
\begin{align*}
\Big| e^{2t}\int_{\R_s^2\times\R_u^2}\bar{f}(t,s,u)\bar{\varphi} (s,u-e^t\bar{u})\mathrm{d}s\mathrm{d}u~-&\int_{\R_u^2} \bar{\varphi}(0,u)\mathrm{d}u \cdot e^{2t}\int_{\R_y^2}\bar{f}(t,s,e^t \bar{u})\mathrm{d}s\Big|\\
&\lesssim \frac{1}{e^{t}}\sup_{(s,u)\in\R^2_s\times \R^2_u}\langle s \rangle^{4}(|\nabla_{s}\bar{\varphi}||\bar{f}|+|Z\bar{f}||\bar{\varphi}|).
\end{align*}
We conclude by using the dominated convergence theorem and Corollary \ref{cor_limit_mass_stable_mfld_origin}.
\end{proof}

\subsection{Conservation laws of the system}

In this subsection, we show an explicit characterization of the total mass and the Hamiltonian energy of the system in terms of the scattering state. In other words, we relate the Hamiltonian energy and the total mass of the initial data $f_0$, to the asymptotic Hamiltonian energy and the asymptotic total mass in terms of the scattering state $f_{\infty}$, respectively.

\begin{proposition}\label{proposition_conservation_total_energy}
For every regular small data solution $f$ to the Vlasov--Poisson system with the potential $\frac{-|x|^2}{2}$, we have $$\mathcal{H}[f_{\infty}]:=\dfrac{1}{2}\int_{\R_x^2\times \R^2_v}(|v|^2-|x|^2)f_{\infty}\mathrm{d}x\mathrm{d}v-\dfrac{\mu}{2}\int_{\R^2_x}|\nabla_u\phi_{\mathrm{asymp}}[Q_{\infty}] |^2\mathrm{d}u=\mathcal{H}[f_0].$$
\end{proposition}

Note that the asymptotic Hamiltonian energy $\mathcal{H}[f_{\infty}]$ is finite, due to the estimates shown in the previous sections. Proposition \ref{proposition_conservation_total_energy} follows by the conservation in time of the Hamiltonian energy, and the following two lemmata. 

\begin{lemma}
There holds $$\lim_{t\to\infty}\int_{\R_x^2\times \R^2_v}(|v|^2-|x|^2)f\mathrm{d}x\mathrm{d}v=-8\int_{\R_s^2\times \R^2_u}(u\cdot s)\bar{f}_{\infty}\mathrm{d}s\mathrm{d}u.$$
\end{lemma}

\begin{proof}
Moving to the hyperbolic variables $(s,u)$ and performing the change of variables $$(s,u)\mapsto (y,w):=\Big(e^ts-\frac{1}{2}\mu t \nabla_u \phi_{\mathrm{asymp}}[Q_{\infty}]\Big(\frac{u}{e^{t}}\Big),\frac{u}{e^{t}}\Big),$$ we have 
\begin{align*}
\int_{\R_x^2\times \R^2_v}(|v|^2-|x|^2)f(t,x,v)\mathrm{d}x\mathrm{d}v&=-8\int_{\R_s^2\times \R^2_u}e^{-t}u\cdot e^{t}s\bar{f}(t,s,u)\mathrm{d}s\mathrm{d}u\\
&=-8\int_{\R^2_y\times \R^2_w}(e^tw)\cdot S_{\C}(t,y,w)\bar{f}(t,S_{\C}(t,y,w),e^tw)\mathrm{d}y\mathrm{d}w,
\end{align*}
where we note that $|(e^tw)\cdot S_{\C}(t,y,w)|\lesssim |w|(|s|+\e(1+t))$. We then deduce that $$\Big|\int_{\R^2_s\times \R^2_u}(u\cdot s)(\bar{f}(t,s,u)-\bar{f}_{\infty}(s,u))\mathrm{d}s\mathrm{d}u\Big|\leq \sup_{(s,u)\in \R^2_s\times\R^2_u}\langle s,u \rangle^4|\bar{f}(t,S_{\C}(t,s,u),e^tu)-\bar{f}_{\infty}(s,u)|,$$ which, in view of $M\geq 6$ and Proposition \ref{prop_mod_scattering_statem_high_reg_proof}, implies the result.

\end{proof}

\begin{lemma}
There holds $$\lim_{t\to\infty}\int_{\R^2_x}|\nabla_x \phi|^2(t,x) \mathrm{d}x=\int_{\R^2_u}|\nabla_u\phi_{\mathrm{asymp}}[Q_{\infty}]|^2(u)\mathrm{d}u.$$
\end{lemma}

\begin{proof}
Performing the change of variables $x\mapsto u=e^{-t}x$, we have $$\int_{\R^2_x}|\nabla_x \phi|^2(t,x) \mathrm{d}x=\int_{\R^2_u}|e^t\nabla_x \phi|^2(t,e^tu) \mathrm{d}u.$$ We then conclude by applying the dominated convergence theorem and Proposition \ref{cor_force_field_along_lineartrajectories}. 
\end{proof}

We finish this subsection with the explicit characterization of the total mass of the system in terms of the scattering state. 

\begin{proposition}\label{proposition_conservation_total_mass}
For every regular small data solution $f$ to the Vlasov--Poisson system with the potential $\frac{-|x|^2}{2}$, we have $$\|f_{\infty}\|_{L^1_{x,v}}=\|f_{0}\|_{L^1_{x,v}}.$$
\end{proposition}

\begin{proof}
Moving to the hyperbolic variables $(s,u)$, and performing the change of variables $$(s,u)\mapsto (y,w):=\Big(e^ts-\frac{1}{2}\mu t \nabla_u \phi_{\mathrm{asymp}}[Q_{\infty}]\Big(\frac{u}{e^{t}}\Big),\frac{u}{e^{t}}\Big),$$ we have $$\int_{\R^2_x\times \R^2_v}f(t,x,v)\mathrm{d}x\mathrm{d}v=2\int_{\R^2_s\times \R^2_u}\bar{f}(t,s,u)\mathrm{d}s\mathrm{d}u=2\int_{\R^2_y\times \R^2_w}\bar{f}(t,S_{\C}(t,y,w),e^tw)\mathrm{d}y\mathrm{d}w.$$ We then deduce that $$\Big|\int_{\R^2_s\times \R^2_u}\bar{f}(t,s,u)-\bar{f}_{\infty}(s,u)\mathrm{d}s\mathrm{d}u\Big|\leq \sup_{(s,u)\in \R^2_s\times\R^2_u}\langle s,u \rangle^3|\bar{f}(t,S_{\C}(t,s,u),e^tu)-\bar{f}_{\infty}(s,u)|,$$ which, in view of $M\geq 6$ and Proposition \ref{prop_mod_scattering_statem_high_reg_proof}, implies the result.
\end{proof}

\begin{remark}
Observe that the total mass $\|f\|_{L^1_{x,v}}$ of the system can also be written as $2\|Q_{\infty}\|_{L^1_{u}}$, since the total mass is equal to twice the total mass of the scattering state by Proposition \ref{proposition_conservation_total_mass}.
\end{remark}


\bibliographystyle{alpha}
\bibliography{Bibliography.bib} 

\end{document}